\documentclass[reqno,11pt]{amsart}

\newif\ifdetails\detailsfalse
\newif\ifoutlines\outlinesfalse
\newif\ifcomments\commentstrue
\newif\iftodos\todostrue
\newif\iflists\listsfalse

\usepackage[utf8]{inputenc}
\usepackage{amsmath} 
\usepackage[mathscr]{eucal}
\usepackage{amssymb}
\usepackage{srcltx} 
\usepackage{dsfont}
\usepackage{hyperref}
\usepackage{color}
\usepackage{enumerate}
\usepackage{tikz}
\usepackage{comment}
\usepackage{lscape}
\usepackage{graphicx}
\usepackage{nicefrac}

\usepackage{glossaries}
\include{symbolverzeichnis1}
\makeglossaries

\usepackage{thmtools}

\numberwithin{equation}{section}
 
\def\cM{{\mathcal M}} 

\newfam\Bbbfam 
\font\tenBbb=msbm10 
\font\sevenBbb=msbm7 
\font\fiveBbb=msbm5 
\textfont\Bbbfam=\tenBbb 
\scriptfont\Bbbfam=\sevenBbb 
\scriptscriptfont\Bbbfam=\fiveBbb

\def\2{\mathbf 2}

\newcommand{\R}     {\mathbb{R}}

\newcommand{\E}     {\mathbb{E}}

\def\1{{\mathchoice {1\mskip-4mu\mathrm l}      
		{1\mskip-4mu\mathrm l} 
		{1\mskip-4.5mu\mathrm l} {1\mskip-5mu\mathrm l}}} 
 
\def\comment#1{} 
\newtheoremstyle{thm}{2ex}{2ex}{\itshape\rmfamily}{} 
{\bfseries\rmfamily}{}{1.7ex}{} 

\newtheoremstyle{rem}{1.3ex}{1.3ex}{\rmfamily}{} 
{\itshape\rmfamily}{}{1.5ex}{}

\newtheorem{theorem}{Theorem}[section] 
\newtheorem{lemma}[theorem]{Lemma} 
\newtheorem{prop}[theorem] {Proposition} 
\newtheorem{cor}[theorem]  {Corollary}

\theoremstyle{definition}
\newtheorem{defn}[theorem] {Definition}

\newtheorem{remark}[theorem]{Remark}

\newcommand{\s}{\sigma}

\newcommand{\Hcal}   {{\mathcal H }}

\definecolor{Red}{rgb}{1,0,0}

\begin{document} 

\def\act{1}
\def\cact{c}
\def\dor{0}
\def\cdor{\tilde{c}}
\def\Zooeps{Z_t^\varepsilon}
\def\MFE{\mathcal{M}_F}
\def\Zoo{Z}
\def\smalltimeinterval{\Delta}
\def\adstate{\sigma}
\def\ooBMtrans{H^\text{\%}}
\def\sbm{X}
\def\borelsets{\mathcal{B}}

	\title[On/off super-Brownian motion]{On/off super-Brownian Motion:\\ Characterization, construction and long-term behaviour}
	\maketitle %

	\thispagestyle{empty}
	\vspace{-0.5cm}
	
	\centerline{\sc Jochen Blath{\footnote{{\tt blath@math.uni-frankfurt.de}}} and Dave Jacobi{\footnote{{\tt jacobi@math.tu-berlin.de}}}}
	\renewcommand{\thefootnote}{}
	\vspace{0.5cm}
	\centerline{\textit{Goethe-Universität Frankfurt and Technische Universität Berlin}}
	
	\bigskip
	
	\vspace{.5cm} 
	
	\begin{quote} 
		{\small {\bf Abstract:}} We introduce and construct on/off super-Brownian motion (on/off SBM) as a measure-valued scaling limit of critical on/off branching Brownian motions. The distinguishing feature of this process is that its infinitesimal particles can switch individually into and out of a state of {\em dormancy}, in which they neither move nor reproduce. Related dormancy traits have received interest in mathematical population biology recently, introducing memory and delay (in the form of a seed bank) into the corresponding processes. 
		It turns out that the properties of on/off SBM differ significantly from those of classical super-Brownian motion. In particular, the process does not die out in finite time with probability one despite criticality of reproduction. However, the size of the active subpopulation does hit 0 in finite time with positive probability, a result which can be shown using methods from polynomial diffusion theory. 
		We expect distinct behaviour in various other qualitative and quantitative respects, for which we provide heuristics, and conclude with a brief discussion of several directions for future research.
	\end{quote}
	
	
	

	\setcounter{tocdepth}{3}
	
	
	\setcounter{section}{0}
	\begin{comment}{
	This is not visible.}
	\end{comment}
	
\iflists	
\tableofcontents
\fi	
	
	\section{Introduction}
	\label{sec-introduction}

	\subsection{Motivation}
	{\em Super-Brownian motion} (SBM, also known as Dawson-Watanabe process) is a measure-valued branching process introduced by Watanabe \cite{W68} and considered in the context of stochastic evolution equations by Dawson \cite{D75}. 
	Jointly with the Fleming-Viot process \cite{FV79} it is one of the prototypical examples of a measure-valued diffusion. We refer to the monograph by Etheridge \cite{E00} for an introduction to the subject and further literature.   
	
	Despite its presence in the mathematical theory for more than 50 years, super-Brownian motion and its relatives still provide a rich source for  research questions, for example regarding their role as scaling limits in models from population genetics, see e.g.\ \cite{CDE18, CP20, PR21}. 
	
	Super-Brownian motion has originally been constructed as a weak limit of the empirical distributions of critical binary {\em branching Brownian motion} as the number of particles (and their branching rate) goes to infinity (see e.g.\ \cite{D93} for a comprehensive account). 
	Branching Brownian motion is of course itself an interesting object and comes in various guises. For example, one natural (supercritical) variant arises as Markov dual of the Fisher-KPP equation from population biology (or chemistry), where it can be used to characterize the critical speed of travelling wave solutions. Indeed, the latter corresponds to the asymptotic speed of the rightmost particle in branching Brownian motion \cite{1983_Bramson_ConvergenceOfSolutionsOfTheKolmogorovEquationToTravellingWaves, 1978_Bramson_MaximalDisplacementOfBranchingBrownianMotion, 1975_McKean_ApplicationOfBrownianMotionToTheEquationOfKolmogorovPetrovskiiPiskunov}. 
	
	Recently, the effects of {\em dormancy} (that is, the ability of individuals to switch between active and inactive states) have gained some attention in population biology, where it leads to the emergence of {\em seed banks} (i.e.\ pools of dormant individuals). Such seed banks can have a strong impact on the long-term behaviour of the underlying systems, introducing memory, resilience and diversity, see e.g.\ \cite{BGKW16, LHWB21}.
	In the context of the Fisher-KPP equation, the introduction of a dormant state resp.\ a seed bank leads to a new dual, namely a supercritical  {\em on/off branching Brownian motion} \cite{BHN22}. While the `on'-state corresponds to particles undergoing the classical dynamics of branching Brownian motion, namely reproduction and movement in space, these are switched off in the `off'-state. The presence of off-states reduces the critical wave-speed in a quantifiable way, as has been shown in \cite{BHJN22+}. 
	
	It thus appears natural, and is the starting point of the present paper, to incorporate an on/off mechanism into the approximating critical binary branching Brownian motion of SBM, and to investigate the corresponding scaling limit in order to make sense of the notion of an {\em on/off super-Brownian motion}. Note that dormancy introduces individual `delays' into the underlying Brownian motions, and the dormant particles at each given time can again be seen as a `seed-bank'. We now describe the corresponding processes.
	
	\begin{defn}[Critical binary on/off branching Brownian motion (on/off BBM)]
	The dynamics of critical binary on/off branching Brownian motion has the following ingredients:
	\begin{itemize}	
		\item Each particle carries one of two states, either active ($\act$) or dormant ($\dor$).
		\item Active individuals ($\act$) move around in $\R^d$ as independent Brownian motions.
		\item With rate $\gamma>0$, an active particle is independently affected by a critical binary branching event. That is, with probability 1/2, it either splits into two new independent active particles, or dies.
		\item Active type ($\act$) individuals switch independently into the dormant state ($\dor$) with exponential rate $\cact$.
		\item Dormant individuals ($\dor$) neither move nor reproduce.
		\item Dormant type ($\dor$) individuals switch into the active state with rate $\cdor$, again independently of all other events.
	\end{itemize}

	We describe on/off BBM by its empirical measure-valued process $\Zoo$ on $\R^d \times \{0,1\}$, defined at times $t \ge 0$ by
\begin{equation*}
	Z_t := \sum_{k = 1}^{n(t)} \delta_{(X_k(t), \s_k(t))} \in \MFE(\R^d \times \{0,1\}).
\end{equation*}
Here, $n(t)$ is the number of particles present at time $t \ge 0$, $k$ refers to an arbitrary enumeration of the particles, $X_k(t)$ gives the spatial position of particle $k$ in $\R^d$, and $\s_k(t)$ refers to its type from $\{0,1\}$ (i.e.\ active or dormant). Finally $\cM_F(\R^d \times \{0,1\})$ denotes the space of finite measures on $\R^d \times \{0,1\}$. In what follows we consider on/off BBMs started from a single particle, say at $(x,i)$, as well as started from a random initial configuration given by the points of a Poisson random measure on $\R^d \times \{0,1\}$ with suitable finite intensity measure $\mu$. For a more explicit definition of this pre-limiting system, we refer to Section \ref{sec-construction}.

\end{defn}

	\begin{remark}\
		i) We emphasize that the present version of on/off BBM is different from the one appearing as the dual of the Fisher-KPP equation with seed bank investigated in \cite{BHN22} and \cite{BHJN22+}: In our case, reproduction is critical and binary, whereas in the F-KPP set-up,
		reproduction events are always supercritical and lead to an increase of the number of active particles by one.\\
		ii) Obviously, many variants of the above system can be considered, including more general branching and motion processes \cite{Dynkin1994}, multi-type versions \cite{GorostizaLopezMimbela1990} with dormancy, coordinated switching between on- and off-states (as in \cite{GKT21, BGKW20}), `deep' seed banks leading to heavy-tailed dormancy times as in \cite{GHO22}, etc., but we wish to keep things as natural as possible in this introductory paper.
	\end{remark}
	
	Our first goal is to derive an interesting limit of the above on/off BBM under a suitable rescaling. Note that, regarding the additional switching parameters $c$ and $\tilde c$, at least two `natural' scaling limits can be considered.
	
	Indeed, for $\epsilon > 0$, we may start with initial distribution given by a Poisson random measure with intensity $\frac{\mu}{\epsilon}$, and consider the rescaled reproduction rates $\frac{\gamma}{\epsilon}$ as well as reducing the mass of each particle to $\epsilon$. This is the classical rescaling for convergence of BBM into SBM. Then it is a modeling choice, whether to scale the switching rates with $\epsilon$ as well or not. The interesting case is the one where the switching rates are left unchanged (i.e.\ do not depend on $\varepsilon$). In this scenario we will be interested in the weak limit of
	\begin{equation}
	\label{eq:empirical_measure}
	(\Zooeps)_{t \ge 0} \quad \mbox{ as } \quad  \varepsilon \to 0
	\end{equation}
	on a suitable path-space.
	We will see that this limit, which we call on/off super-Brownian motion, can be regarded as a natural example for a non-local superprocess in the sense of Dynkin \cite{Dynkin1994} and is closely related to the  multi-type measure branching processes of \cite{GorostizaLopezMimbela1990}. However, due to the dormant states in which both reproduction and motion are turned off,  this measure-valued process can be expected to behave significantly different from classical super-Brownian motion (and also its relatives from \cite{GorostizaLopezMimbela1990}, where motion is always `on').
	
	The other natural modeling choice would be to rescale the switching parameter in the same fashion as the other mechanisms, i.e.\ to consider the scaled switching rates $c/\varepsilon, \tilde c/\varepsilon$, but the corresponding scaling limit will then not lead to novel behaviour. What one expects to see is a classical super-Brownian motion with `effective parameters' obtained from the stationary distribution of the two-state Markov chain with switching rates $c$ and $\tilde c$. Here, dormancy will merely `slow down' motion and reproduction. The effective branching rate of the limiting object can be expected to be $\tilde \gamma = \frac{\tilde c}{c+\tilde c} \gamma$,  and the underlying motion process should be given by a `delayed' Brownian motion with variance $\frac{ \tilde c}{c+\tilde c} t$ at time $t$. This latter object is of course again a (time-changed) SBM and thus entirely known. Hence we now turn our attention to the new limit obtained from $\{Z_t^{\varepsilon}\}_{t \ge 0}$ as $\varepsilon \to 0$.

\subsection{On/off super-Brownian motion: Definition and main results}

We now introduce and describe our main object of study. Let $bp \mathcal{B}(\R^d \times \{0,1\})$ denote the bounded Borel measurable positive functions and $\cM_F(\R^d \times \{0,1\})$ the finite measures. By $\langle X_t , \phi \rangle$ we denote the integral $\int \phi(x,i) dX_t(x,i)$.

\begin{theorem}[On/off super-Brownian Motion: Existence and characterization]
Let $\mu$ be a finite measure on $\R^d \times \{0,1\}$. Then there exists an $\cM_F(\R^d \times \{0,1\})$-valued branching Markov process $(X_t)_{t \ge 0}$, starting from $X_0=\mu$, whose transition probabilities are determined by the functional Laplace transform
\begin{equation}
	\label{eq:ooSBMduality}
\E_\mu [e^{-\langle X_t, \phi \rangle}] = \exp(- \int_{\R^d \times \{0,1\}} V_t \phi(x,i) \, d\mu(x,i)), \quad \phi \in bp \mathcal{B}(\R^d \times \{0,1\}),
\end{equation}
where $(V_t)_{t\ge 0}$ solves the system of evolution equations
\begin{align}
	\label{eq:ooSBMPDEdual1}
V_t \phi(x,1) &= \E_{(x,1)}[\phi(B_t,1) + c   \int_0^t V_{t-s} \phi(B_s,0) - V_{t-s} \phi(B_s,1) \, ds  \\
					&\qquad\qquad\qquad\qquad -  \frac \gamma 2  \int_0^t (V_{t-s} \phi(B_s,1))^2 \, ds], \notag \\ 
V_t \phi(x,0) &= \phi(x,0) + \tilde c \int_0^t V_{t-s} \phi(x,1) - V_{t-s} \phi(x,0) \, ds,
\label{eq:ooSBMPDEdual2}
\end{align}
with $B$ being a standard Brownian motion.
\end{theorem}
\begin{remark}
	Note that Equation \eqref{eq:ooSBMduality} is a Markov process duality.
\end{remark}	
\begin{proof}
	This is a direct consequence of Theorem \ref{thm:ooSBMExistencewithpsi} and Equations \eqref{def:abcoeffs} and \eqref{def:Gammacoeff}.
\end{proof}

The on/off SBM is indeed the scaling limit of critical binary branching on/off Brownian motion as considered in Equation \eqref{eq:empirical_measure}:

\begin{theorem}[Convergence of on/off BBM to on/off SBM]
	\label{thm:convergence_on_path_space}
	Let $\mu, \gamma, c, \tilde c$ be fixed and $(Z_t^\varepsilon)_{t \ge 0}$ as above. Let $(X_t)_{t \ge 0}$ be an on/off super-Brownian motion with $X_0=\mu$. Then
	$$
	(Z_t^\varepsilon)_{t \ge 0} \to (X_t)_{t \ge 0}
	$$
	weakly on the space of càdlàg paths endowed with the Skorokhod topology as $\epsilon \to 0$.
\end{theorem}
\begin{proof}
	This follows from the convergence of the finite dimensional marginals of on/off branching Brownian motion to the finite dimensional marginals of on/off super-Brownian motion, the path continuity of on/off super-Brownian motion (Proposition \ref{prop:pcoosbm}) and the tightness on the path-space of on/off branching Brownian motion (Theorem \ref{thm:tpsoobbm}).
\end{proof}

There are further characterizations available. For the corresponding martingale problem representation see Proposition \ref{prop:mpfoosbm}, and for the generator (on suitable test functions) see Theorem \ref{thm:goooSBM}.

Note that by Theorem \ref{thm:convergence_on_path_space}, we have that the paths of on/off SBM are càdlàg. Since classical SBM has continuous paths, and the switching mechanism happens independently on an individual level (as opposed to in a coordinated fashion), we expect the same to hold for on/off SBM, and in fact this is true.

\begin{theorem}[Path continuity]
	\label{thm:path-cont}
	On/off SBM has a version with continuous paths in $\cM_F(\R^d \times \{0,1\})$.
\end{theorem}

See Proposition \ref{prop:pcoosbm} for a proof.

The above results can be derived from by now standard superprocess machinery, following e.g.\ the rather well-trodden paths in \cite{D93}, \cite{GorostizaLopezMimbela1990}, \cite{Dynkin1994}, \cite{1986_RoellyCoppoletta_ACriterionOfConvergenceOfMeasureValuedProcessesApplicationToMeasureBranchingProcesses}. However, the literature does not always provide full details or results tailored to our specific case. Since these can  be somewhat  cumbersome to work out, we decided
to provide the particularities in Section \ref{sec-construction} in a hopefully suitable fashion, for convenience.

In the next section, we turn to the long-term properties of our process, which are specific to our model and exhibit novel types of behaviour. Our methods will then make use of another feature of on/off SBM, namely its polynomiality.

\subsection{Survival and long-term behaviour}

Recently, the notion of a {\em polynomial diffusion} has been introduced in \cite{2016_FilipovicLarsson_PolynomialDiffusionsAndApplicationsInFinance} and extended to the measure-valued case in \cite{CLS19, 2021_Cuchiero_MeasureValuedAffineAndPolynomialDiffusions}. Our on/off SBM provides a further natural example for this class of processes. This allows the application of polynomiality-related techniques especially concerning the boundary behaviour of its (two-dimensional) total mass processes, as we will see below. In particular, this will allow us to circumvent the restrictions of the classical Feller boundary classification machinery to one-dimensional diffusions.

\begin{prop}[Polynomial measure-valued diffusion]
On/off SBM is a {\em polynomial measure-valued diffusion} in the sense of \cite{2021_Cuchiero_MeasureValuedAffineAndPolynomialDiffusions}. In particular dormancy and the resulting non-local branching character of on/off super-Brownian motion do not break its polynomiality.
\end{prop}

This follows by inspection from Definitions 6.1 and 6.3 in \cite{2021_Cuchiero_MeasureValuedAffineAndPolynomialDiffusions} and our generator characterization in Theorem  \ref{thm:goooSBM}.
In what follows, it will be useful to consider the active and the dormant sub-population of on/off SBM separately.

\begin{remark}[The active and the dormant component]
There are two equivalent ways of describing the state-space of on/off super-Brownian motion, which is based on the homeomorphism
\begin{equation*}
\mathcal{M}_F(\mathbb{R}^d \times \{0,1\}) \simeq \mathcal{M}_F(\mathbb{R}^d)^2,
\end{equation*}
see \cite{2002_DawsonGorostizaLi_NonlocalBranchingSuperprocessesAndSomeRelatedModels}. Where suitable, we will resort to the description of on/off super-Brownian motion as $(X_t^a,X_t^d)$ taking values in $\mathcal{M}_F(\mathbb{R}^d)^2$.
\end{remark}

The {\em survival and extinction behaviour} of on/off SBM is non-classical, due to the dormant subpopulation, which forms a non-vanishing `seed bank'. To investigate this, we first derive the dynamics of the total mass processes of on/off SBM. 

\begin{prop}[Total mass processes]
	\label{prop:total_mass}
    The total mass process of the active and the dormant component of on/off SBM, denoted by 
    $(\langle X^a_t, 1 \rangle, \langle X^d_t, 1\rangle)_{t \ge 0}$,
    is equal in distribution to the solution $(p_t,q_t)_{t\ge 0}$ of the system of stochastic differential equations
    \begin{align} 
        d p_t = & (\tilde{c} q_t-c p_t)dt +\frac{\gamma}{2}\sqrt{p_t}dB_t, \notag\\
        d q_t = & (c p_t - \tilde{c} q_t) dt,
		\label{eq:F-KPP:with_seed_bank_original}
    \end{align}
	with initial conditions $p_0=\langle X^a_0, 1\rangle$ and $q_0=\langle X^d_0, 1\rangle$, where $(B_t)_{t\ge 0}$ is a standard Brownian motion. 
    Here, $(p,q)$ is called on/off Feller diffusion.
    Moreover, the total mass process $(\langle X_t, 1 \rangle)_{t\ge 0}$ is equal in law to the process $(r_t)_{t\ge 0}$ defined by 
    $r_t := p_t + q_t$ for $t\ge0,$ with initial conditions given as before. 
\end{prop}
See Proposition \ref{prop:tmoosbm} for a proof.
From a comparison argument for the dormant component, we obtain the following persistence property.
\begin{theorem}[Long-term persistence of on/off SBM]
	\label{cor:longtime-behaviour}
Contrary to classical SBM, on/off SBM never dies out in finite time. However, its total mass converges almost surely to $0$.  
\end{theorem}

\begin{proof}
	This is the combination of Corollary \ref{cor:apm} and Proposition \ref{prop:tmtz}.
\end{proof}

This is an example of the increased resilience of populations in the presence of a seed bank who would otherwise die out a.s.\ in finite time.
Interestingly, the active population may still hit zero at certain times with positive probability, despite the constant mass influx from the seed bank. 

\begin{theorem}[Extinction of the active population at a finite time]
	\label{cor:longtime-extinction}
	The total mass process of the active component of on/off SBM hits zero in finite time with positive probability.
\end{theorem}

\begin{proof}
	This follows by an application of Theorem 5.7 in \cite{2016_FilipovicLarsson_PolynomialDiffusionsAndApplicationsInFinance}, see (the proof of) Proposition \ref{prop:achz} and Corollary \ref{cor:acehzwpp} in Section 3.
\end{proof}
\begin{remark}
	On/off super-Brownian motion is a canonical process with a somewhat peculiar extinction behaviour. In particular it may serve as an example that illustrates some questions from Section 2 (p.\ 3470) of \cite{2018_KyprianouPalau_ExtinctionPropertiesOfMultiTypeContinuousStateBranchingProcesses}. For on/off SBM we see extinction as a result of total mass limiting to zero, but remaining positive for all time, while with some positive probability the active mass can go extinct after a finite amount of time, although the transition matrix of the two-type Markov chain is irreducible.
\end{remark}
\subsection{Heuristics and outlook on future research}
\label{sec-outlook} 

We expect the presence of a dormant sub-population / seed bank to produce further novel behaviour in on/off SBM, in particular regarding its support and range properties, as we will indicate below. However, in order to tackle such questions, the technology to analyse our process needs to be developed further, and this is beyond the scope of this introductory paper (and will be part of future research).

\begin{remark}[Some heuristics for the range and the support properties of on/off SBM]
	i) On the one hand we expect that the {\em support properties} of on/off SBM are qualitatively different from classical SBM (both, when comparing the supports of the active and dormant components separately as well as their union). The intuitive reason is that all sites that are being `visited' by on/off SBM should immediately be home to a small dormant sub-population, which will never fully vanish anymore (though it may be decaying exponentially to zero). This should lead to a monotonously growing support (in time), which is very different to the support process of SBM which can be described as a `coherent wandering random measure, see \cite{DH82}. 
	
	ii) On the other hand we expect that the {\em range} of on/off SBM will agree with the range of classical SBM. The heuristic reason is that the approximating empirical measures of on/off SBM and SBM can be coupled to visit exactly the same sites over the whole time interval $[0, \infty)$, just not at the same times. The point is that the actively moving parts can be coupled using the same driving Brownian motions, while during dormancy, particles do not move, and thus do not extend the range. So although the approximating on/off BBM `lives much longer', it will actually visit the same sites as a coupled classical BBM. Properties such as charging sets as in  \cite{I88} should thus be the same for SBM and on/off SBM.
\end{remark}

The definition and basic properties of on/off SBM presented above invite a series of follow up research questions. Here, we sketch a few of them:

A natural first task consists of the the derivation of a look-down construction for the on/off super-Brownian motion à la \cite{DK99, 2019_EtheridgeKurtz_GenealogicalConstructionsOfPopulationModels, 2011_KurtzRodriguez_PoissonRepresentationsOfBranchingMarkovAndMeasureValuedBranchingProcesses} providing an a.s.\ construction for on/off SBM and the approximating on/off BBMs on the same probability space. This might also lead to an almost sure coupling of the ranges of the limiting objects, formalizing the above range heuristics.

Further, an SPDE representation should be derived for on/off SBM in dimension one. While most of the standard arguments should carry over to this case, moment estimates as required in the approach in \cite{KS88} may be harder to obtain recursively, since switching events in the recursion will not reduce the complexity of underlying tree structures. With the help of the SPDE representation, a delay-SPDE reformulation à la \cite{BHN22} can be derived, and this should provide further understanding for the growth properties of the support, in particular in the dormant component.

Another direction of research would be to try to derive an {\em on/off Brownian snake} à la Le Gall \cite{1991_LeGall_BrownianExcursionsTreesAndMeasureValuedBranchingProcesses, 2002_DuquesneLeGall_RandomTreesLevyProcessesAndSpatialBranchingProcesses, 2005_DuquesneLeGall_ProbabilisticAndFractalAspectsOfLevyTrees}, and to investigate potential links of on/off SBM with continuum random trees. Backbone representations,
(fractal) support properties, the notion of a wandering random measure and so on also fall into this research direction.

A fourth area for research would be to construct the corresponding notion of an on/off Fleming-Viot process (including a corresponding lookdown-construction), and to investigate its properties as well as the relation between both processes. One specific goal could be to derive disintegration-theorems à la Perkins \cite{Perkins1992}.

Finally, many generalizations of our setup can be considered including different motion/mutation processes, more general (non-local) reproductive mechanisms (in particular skewed ones), multiple types and switching mechanisms, ``deep'' seed banks with heavy-tailed wake-up times as in \cite{GHO22}, or the interplay with random environments and rough on/off SBM à la \cite{PR21}, and coordinated switching mechanisms as in \cite{BGKW20}, \cite{GKT21}.

\subsection{Notation}
In this subsection we quickly list the various function classes that we are going to use through out this paper.
\begin{itemize}
	\item $b\mathcal{C}(\mathbb{R}^d \times \{0,1\})$, continuous bounded functions.
	\item $bp\mathcal{C}(\mathbb{R}^d \times \{0,1\})$, continuous bounded positive functions.
	
	\item $\mathcal{B}(\mathbb{R}^d \times \{0,1\})$, Borel sets.
	\item $b\mathcal{B}(\mathbb{R}^d \times \{0,1\})$, bounded Borel measurable functions.
	\item $p\mathcal{B}(\mathbb{R}^d \times \{0,1\})$, positive Borel measurable functions.
	\item $bp\mathcal{B}(\mathbb{R}^d \times \{0,1\})$, bounded positive Borel measurable functions.
	\item $b_Ap\mathcal{B}(\mathbb{R}^d \times \{0,1\})$, positive Borel measurable functions bounded by the constant $A$.
	
	\item $\mathcal{M}_F(\mathbb{R}^d \times \{0,1\})$, finite measures.
	\item $\mathcal{M}_1(\mathbb{R}^d \times \{0,1\})$, probability measures.
\end{itemize}
By $\langle \mu , \phi \rangle$ we denote the integral $\int \phi d\mu$ and by $\delta_{(x,i)}$ we denote the dirac measure in $(x,i)$.
\section{Construction and characterization of on/off SBM}
	\label{sec-construction} 

In this section, we first derive the (functional) Laplace transform of the empirical process of the approximating on/off BBM, starting in a single particle, and derive the corresponding evolution equation for the rescaled process (Section \ref{ssn:Laplace_ooBBM}). Then we consider the solution theory for a class of evolution equations for non-local branching mechanisms, which turn out to be suitable for our case, and establish existence and uniqueness of the solution of the limiting evolution equation obtained from rescaling (Section \ref{ssn:evoeq_soln}). Further we establish the convergence of the one- and subsequently finite dimensional marginal distributions of on/off BBM by the the corresponding convergence of the rescaled functional Laplace transforms (Section \ref{ssn:fdd}). Finally, in Section \ref{sec:tightness}, we establish tightness of the laws of the rescaled empirical on/off BBMs on the space of Skorokhod paths, provide a martingale problem representation of the limiting on/off SBM, and establish path continuity of this measure-valued process. To carry out this program, we follow the general strategy for the construction of measure-valued processes from empirical particles as presented e.g.\ in \cite{D93}, combined with theory for non-local branching processes from \cite{Dynkin1994}. 
	\subsection{Laplace transform and evolution equation of the approximating on/off BBM}
	\label{ssn:Laplace_ooBBM}
We consider an on/off branching Brownian motion (with binary branching) given at time $t$ by the empirical measure
\begin{equation*}
	\Zoo_t := \sum_{k = 1}^{n(t)} \delta_{(X_k(t), \adstate_k(t))},
\end{equation*}
where $X_k(t)$ is the position in $\mathbb{R}^d$ of the $k$-th particle at time $t$, $\adstate_k(t)$ is the state ($1$ for active, $0$ for dormant) of the $k$-th particle at time $t$ and $n(t)$ is the total number of particles. We usually assume that on/off BBM with binary branching starts with one particle at position $(x,i)$ at time $t$, i.e.\ $\Zoo_0=\delta_{(x,i)}$.

Recall that on/off BBM (with binary branching) has the following behaviour: Each particle in the active state performs an independent Brownian motion and carries two exponential clocks, one with rate $\gamma > 0$, the branching rate, and one with rate $\cact > 0$, the dormancy initiation rate. If the clock with rate $\cact$ triggers first, then the particle changes its state from active ($1$) to dormant ($0$). If the clock with rate $\gamma$ rings first, then the particle undergoes a branching event, where with probability $\frac{1}{2}$ it either splits into two active particles, that then perform as independent copies of on/off BBM, or with probability $\frac{1}{2}$ the particle will vanish (`die').
Dormant particles neither move nor reproduce, but instead stay at their respective position until they `wake up' (with rate $\cdor$).
When this happens, the particle changes its state from dormant ($0$) to active ($1$) and starts performing an independent copy of on/off BBM.

To identify the corresponding switching times and events, we introduce the following notation. Let
\begin{equation*}
T := \inf \big\{ t \geq 0 \, | \, n(t-) \neq n(t) \big\}
\end{equation*}
be the first branching time of the initial particle and
\begin{equation*}
J := \inf \big\{ t \geq 0\, | \, \adstate_1(t-) \neq \adstate_1(t) \big\}
\end{equation*}
be its first type switching time. Let
\begin{equation*}
\tau_1 := \min \{ T,J \}
\end{equation*}
be the time of the first event.
We define the functional Laplace transform of the empirical measure at time $t\ge 0$, starting at state $(X_1(0), \sigma_1(0))=(x, j)$, to be
\begin{equation*}
w_t\phi(x,j) := \mathbb{E}_{(x,j)} \big[\exp(-\langle \Zoo_t, \phi \rangle)\big],
\end{equation*}
for test functions $\phi \in p\mathcal{B}(\mathbb{R}^d \times \{0,1\})$. We use the shorter notation of $\mathbb{E}_{(x,j)
}[\cdot]$ for $\mathbb{E}_{\delta_{(x,j)}}[\cdot]$.

\begin{theorem}(Evolution equation for the functional Laplace transform of on/off BBM)
    \label{thm:eeoobbm}
For $\phi$ as above, we have
\begin{align*}
w_t\phi(x,1) &= \mathbb{E}_{(x,1)}\left[e^{-\phi(B_t,1)}e^{-(\gamma + \cact)t} + \int_0^t e^{-(\gamma+\cact)s} \left[ \cact w_{t-s}\phi(B_s, 0) + \gamma\left( \frac{1}{2} + \frac{1}{2} (w_{t-s}\phi(B_s,1))^2 \right) \right] ds \right],\\
w_t\phi(x,0) &= e^{-\phi(x,0)}e^{-\cdor t} + \int_0^t \cdor e^{-\cdor s} w_{t-s}\phi(x,1) \, ds,
\end{align*}
where $(B_t)_{t \geq 0}$ is a {standard} Brownian motion. 
\end{theorem}
\begin{proof}
The proof follows by direct computation: Let $\tau_1$ be the time of the first event. Then, we have
\begin{equation*}
w_t\phi(x,j) := \mathbb{E}_{(x,j)} \big[\exp(-\langle \Zoo_t, \phi \rangle)] = \mathbb{E}_{(x,j)} [\exp(-\langle \Zoo_t, \phi \rangle)(1_{\{t < \tau_1\}} + 1_{\{t \geq \tau_1\}})\big].
\end{equation*}

We begin with the first summand: Since $t < \tau_1$, the process has not branched yet. Therefore there is just one particle of the starting-type present. If $j=1$ (active), the particle moves according to a Brownian motion $B$. Therefore we have that $$\exp(-\langle \Zoo_t, \phi \rangle) = \exp(-\phi(B_t,1)).$$ Since $\tau_1$ is the minimum of two independent exponentially distributed random variables with rates $\gamma$ and $\cact$ respectively, we have that $\tau_1$ is exponentially distributed with rate $\gamma + \cact$. Integrating out $\tau_1$ using the density of the exponential distribution
and 
$$
\int_t^\infty (\gamma+\cact)e^{-(\gamma+\cact)s}ds = 1 - \int_0^t (\gamma+\cact)e^{-(\gamma+\cact)s}ds = 1-(1- e^{-(\gamma+\cact)t}) = e^{-(\gamma+\cact)t},
$$
we get
\begin{equation*}
    \mathbb{E}_{(x,1)} \big[\mathbb{E}_{(x,1)} [\exp(-\langle \Zoo_t, \phi \rangle)1_{\{t < \tau_1\}}\, |\, \tau_1]\big] = \mathbb{E}_{(x,1)} \big[ e^{-\phi(B_t,1)}e^{-(\gamma+\cact)t}\big].
\end{equation*}
Analogously, if $j=0$ (dormant initial state), given $t< \tau_1$, we have that 
$$
\exp(-\langle \Zoo_t, \phi \rangle) = \exp(-\phi(x,0)).
$$ 
Since $\tau_1$ is exponentially distributed with rate $\cdor$ (the wake-up rate), we obtain 
\begin{equation*}
\mathbb{E}_{(x,0)} \big[\mathbb{E}_{(x,0)} [\exp(-\langle \Zoo_t, \phi \rangle)1_{\{t < \tau_1\}}\, |\, \tau_1]\big] =\mathbb{E}_{(x,0)} \big[ e^{-\phi(x,0)}e^{-\cdor t}\big].
\end{equation*}

Now we turn to the second summand.
If $j=1$ (active initial state), the probabilities of the three possibilities that can happen at the time of the first event are
given by
\begin{align}
    \label{eq:eeooBBM1}
    \mathbb{P}[\Zoo_{\tau_1} = \delta_{(B_{\tau_1}, 0)}] &= \frac{\cact}{\gamma + \cact}, \\
    \label{eq:eeooBBM2}
    \mathbb{P}[\Zoo_{\tau_1} = 0] &= \frac{\gamma}{2(\gamma + \cact)}, \\
    \label{eq:eeooBBM3}
    \mathbb{P}[\Zoo_{\tau_1} = \delta_{(B_{\tau_1}, 1)} + \delta_{(B_{\tau_1}, 1)}] &= \frac{\gamma}{2(\gamma + \cact)}.
\end{align}
For $j=1$, we have
\begin{align*}
\mathbb{E}_{(x,1)}\big[&\exp(-\langle \Zoo_t, \phi \rangle) 1_{\{\tau_1 \leq t\}}\big] \\
&= \mathbb{E}_{(x,1)}\left[\left(\frac{\gamma}{2(\gamma+\cact)} + w_{t-\tau_1}\phi(B_{\tau_1}, 0) \frac{\cact}{\gamma+\cact} + w_{t-\tau_1}\phi(B_{\tau_1}, 1)^2 \frac{\gamma}{2(\gamma+\cact)}\right)1_{\{\tau_1 \leq t\}}\right],
\end{align*}
where we used the Markov property, the independence of the switching mechanisms and trigger times, and the branching property.
Integrating over the time of the first event, which is exponentially distributed with parameter $\gamma + \cact$, we get that the above is equal to
\begin{equation*}
    \mathbb{E}_{(x,1)} \left[ \int_0^t e^{-(\gamma + \cact)s} \left[\cact w_{t-s}\phi(B_s, 0)  + \gamma \left( \frac{1}{2} + \frac{1}{2} (w_{t-s}\phi(B_s,1))^2 \right) \right] \, ds \right].
\end{equation*}
For the case $j=0$ (dormant initial state), we have that $\tau_1 \sim \exp(\cdor)$, and just one thing can happen at the switching time, namely that the particle wakes up. This means that necessarily $\Zoo_{\tau_1} = \delta_{(x,1)}$. Therefore we get
\begin{equation*}
    \mathbb{E}_{(x,0)} \big[\exp(-\langle \Zoo_t, \phi \rangle) 1_{\{\tau_1 \leq t\}}\big] = \int_0^t \cdor e^{-\cdor s} w_{t-s}\phi(x, 1) \, ds
\end{equation*}
by the Markov property.
Putting everything together yields the assertion.
\end{proof}
Next, we provide an alternative formulation of the evolution equation for the Laplace functional of on/off Branching brownian motion. This form now exhibits a `migration type' relation between the active and dormant components. 

\begin{theorem}[Alternative formulation of the evolution equation for on/off BBM]
For $\phi$ as before we have	
\label{thm:aeeooBBM}
\begin{align*}
w_t\phi(x,1) &= \mathbb{E}_{(x,1)} \left[e^{-\phi(B_t,1)} + \cact \int_0^t w_{t-s}\phi(B_s,0) - w_{t-s}\phi(B_s,1) ds \right. \\
&\qquad\qquad\qquad\qquad\qquad\left.+ \gamma \int_0^t \left( \frac{1}{2} + \frac{1}{2}(w_{t-s}\phi(B_s,1))^2 - w_{t-s}\phi(B_s,1) \right) \,  ds \right], \\
w_t\phi(x,0) &= e^{-\phi(x,0)} + \tilde{c} \int_0^t w_{t-s}\phi(x,1) - w_{t-s}\phi(x,0) \, ds,
\end{align*}
where $(B_t)_{t \geq 0}$ is a {standard} Brownian motion. 
\end{theorem}

\begin{proof}
This follows from the previous Theorem \ref{thm:eeoobbm} by partial integration and the Markov property as in \cite[Lemma 4.3.4]{D93}.
\end{proof}

The next step is to derive the evolution equation of the rescaled on/off BBM. For $t \ge 0$, let $Z_t^\varepsilon$ be the empirical measure of an on/off BBM with branching rate $\gamma/\varepsilon$, started in a single particle at position $x$ and activity state $j \in \{0,1\}$ and each particle having mass $\epsilon$. As before, we denote by $w_t^\epsilon\phi$ its functional Laplace transform involving the test function $\phi$.

\begin{cor}[Scaled alternative evolution equation for on/off BBM]
With the notation as above, we have
\begin{align*}
w_t^\epsilon\phi(x,1) &= \mathbb{E}_{(x,1)}\bigg[e^{-\epsilon\phi(B_t,1)} + \cact \int_0^t w^\epsilon_{t-s}\phi(B_s,0) - w^\epsilon_{t-s}(B_s,1) ds \\
& \qquad\qquad\qquad\qquad\qquad + \frac{\gamma}{\epsilon}\int_0^t\left( \frac{1}{2} + \frac{1}{2}(w^\epsilon_{t-s}\phi(B_s,1))^2 - w^\epsilon_{t-s}\phi(B_s,1) \right)ds \bigg], \\
w_t^\epsilon\phi(x,0) &= e^{-\epsilon\phi(x,0)} + \cdor\int_0^t w^\epsilon_{t-s}\phi(x,1) - w^\epsilon_{t-s}\phi(x,0) \,ds.
\end{align*}
\end{cor}

\begin{proof}
We use that $w_t^\epsilon\phi = w_t\tilde{\phi}_\epsilon$ for $\tilde{\phi}_\epsilon(x,i) := \varepsilon\phi(x,i)$.
\end{proof}
Define, for $t\ge 0$,
\begin{equation}
    \label{def:vtepsilon}
	v_t^\epsilon\phi(x,i) := \frac{1-w_t^\epsilon\phi(x,i)}{\epsilon}.
\end{equation}
A further transformation of the evolution equation will be useful for the identification of the scaling limit.

\begin{cor}[Evolution equation for $v_t^\epsilon$]
With the above notation, we have
\begin{align*}
v_t^\epsilon\phi(x,1) &= \mathbb{E}_{(x,1)}\Big[ \frac{1 - e^{-\epsilon\phi(B_t,1)}}{\epsilon} + \cact \int_0^t v_{t-s}^\epsilon\phi(B_s,0) - v_{t-s}^\epsilon\phi(B_s,1) \,  ds 
-\frac{\gamma}{2} \int_0^t(v_{t-s}^\epsilon\phi(B_s,1))^2 \, ds \Big], \\
v_t^\epsilon\phi(x,0) &= \frac{1 - e^{-\epsilon\phi(x,0)}}{\epsilon} + \cdor \int_0^t v_{t-s}^\epsilon\phi(x,1) - v_{t-s}^\epsilon\phi(x,0)\, ds.
\end{align*}
\end{cor}

\begin{proof}
We have by definition
\begin{equation*}
w_t^\epsilon\phi(x,i) = 1 - \epsilon v_t^\epsilon\phi(x,i).
\end{equation*}
Plugging this into the alternative evolution equation for the scaled empirical measure gives
\begin{equation*}
    v_t^\epsilon\phi(x,0) = \frac{1 - e^{-\epsilon\phi(x,0)}}{\epsilon} - \frac{\cdor}{\epsilon} \int_0^t\epsilon ( v_{t-s}^\epsilon\phi(x,0) - v_{t-s}^\epsilon\phi(x,1))ds
\end{equation*}
and 
\begin{align*}
v_t^\epsilon\phi(x,1) &= \frac{1}{\epsilon}\Bigg(1 - \mathbb{E}_{(x,1)}\Big[e^{-\epsilon\phi(B_t,1)} + \int_0^t \cact \big( (1 - \epsilon v_{t-s}^\epsilon\phi(B_s,0)) - (1 - \epsilon v_{t-s}^\epsilon\phi(B_s,1) ) \big) \\
& \qquad + \frac{\gamma}{\epsilon}\Big( \frac{1}{2} + \frac{1}{2}(1 - \epsilon v_{t-s}^\epsilon\phi(B_s,1))^2 - (1 - \epsilon v_{t-s}^\epsilon\phi(B_s,1)) \Big) \,  ds \Big]\Bigg).
\end{align*}
Simplification inside the integral yields 
\begin{equation*}
    v_t^\epsilon\phi(x,1) = \mathbb{E}_{(x,1)}\Bigg[ \frac{1 - e^{-\epsilon\phi(B_t,1)}}{\epsilon} + \cact \int_0^t v_{t-s}^\epsilon\phi(B_s,0) - v_{t-s}^\epsilon\phi(B_s,1) ds - \frac{\gamma}{2} \int_0^t(v_{t-s}^\epsilon\phi(B_s,1))^2 \, ds \Bigg]
\end{equation*}
as desired.
\end{proof}
\subsection{Evolution equations: Solution theory for non-local branching mechanisms}
	\label{ssn:evoeq_soln}
	In this section we provide the solution theory for stochastic evolution equations that constitute the PDE-dual of on/off super-Brownian motion in a fashion tailored to our needs. Due to the non-locality of the branching behaviour (where e.g.\ switches into dormancy can be interpreted as offspring production of active individuals not in the local active, but in the `detached'  dormant part of the  statespace), we cannot prove the required results by the usual Gronwall argument, Lemma 4.3.1 in 
	\cite{D93},
	but instead have to rely on the construction of the solution backwards in time on small time intervals as in \cite{Dynkin1994}.

	The main goal of this section will be the proof of the non-local branching solution Theorem \ref{thm:nlbst}. The proof relies heavily on the Lipschitz continuity of the branching mechanism on \emph{bounded} functions, which stems from the \emph{local} Lipschitz continuity of $z \mapsto z^2$. This will be considered in Subsection \ref{ssn:lcosnlbm}. As the proof of the non-local branching solution theorem is a backwards induction on small time intervals, it is useful to reformulate it using notation of time inhomogeneous Markov processes, see Subsection \ref{ssn:ratimp}. In the induction step as well as in the initial step, we will use a Lipschitz continuity trick (Lemma \ref{lem:lct}) and a uniform Cauchy sequence property (Lemma \ref{lem:uniform_cs_property_of_F}), which we prepare in Subsection \ref{ssn:lctaucsp}. In order to apply these, we have to make sure that the approximating solution candidates $v_t^\epsilon\phi$ remain bounded (Proposition \ref{prop:bddvtepsphi}). This circumstance originates again from the mere \emph{local} Lipschitz continuity of the square function. We ensure the required boundedness properties in Subsection \ref{ssn:vbp}. Finally, in Subsection \ref{ssn:nlbst}, we can prove the non-local branching solution Theorem \ref{thm:nlbst}. In the last Subsection \ref{ssn:pos}, we gather properties of the solution that will be useful later.
	
	\subsubsection{Lipschitz continuity of simple non-local branching mechanisms}
	\label{ssn:lcosnlbm}

We define an operator $\psi$ that will contain the non-linearity and the coupling terms of the approximating evolution equations.	It is a special case of the branching operator considered in Equation 4.2 on p.\ 52 of \cite{Dynkin1994}, therefore we call it `simple'.
Let $b\mathcal{B}(\mathbb{R}^d \times \{0,1\})$ denote the space of bounded measurable functions on $\mathbb{R}^d \times \{0,1\}$.
\begin{defn}[Simple non-local branching mechanism $\psi$]
Define an operator $\psi: b\mathcal{B}(\mathbb{R}^d \times \{0,1\}) \to b\mathcal{B}(\mathbb{R}^d \times \{0,1\})$ via
\begin{equation}
    \label{eq:sfopsi1}
\psi(z)(x,i) := a(x,i)z(x,i) + b(x,i)z(x,i)^2 - \int z(y,j) \Gamma((x,i);d(y,j)),
\end{equation}
for some functions $a,b: \mathbb{R}^d \times \{0,1\} \to \mathbb{R}$ and some kernel $\Gamma$ from $\mathbb{R}^d \times \{0,1\}$ to $\mathbb{R}^d \times \{0,1\}$ and $z: \mathbb{R}^d \times \{0,1\} \to \mathbb{R}$. 
\end{defn}
The case of on/off branching Brownian motion then corresponds to the choice
\begin{equation}
\label{def:abcoeffs}
a(x,i) := \begin{cases} \cact, & i= 1, \\
\cdor, & i= 0,
\end{cases}, \quad
b(x,i) := \begin{cases} \frac{\gamma}{2}, & i= 1, \\
0, & i= 0,
\end{cases}
\end{equation}
and
\begin{equation}
\label{def:Gammacoeff}
\Gamma(x,i) := \begin{cases} \cact\delta_{(x,0)}, & i= 1, \\
\cdor\delta_{(x,1)},& i=0.
\end{cases}
\end{equation}
In particular, for all $(x,i) \in \mathbb{R}^d \times \{0,1\}$, we have
\begin{equation*}
|a(x,i)|, |b(x,i)|, |\Gamma(x,i, \mathbb{R}^d \times \{0,1\})| \leq Q,
\end{equation*}
where $Q = \cact + \cdor + \gamma$. We will see in the next Proposition \ref{prop:sclcp}, that therefore $\psi$ is Lipschitz continuous w.r.t.\ the supremum norm on $b_Ap\mathcal{B}(\mathbb{R}^d \times \{0,1\})$ for any $A \geq 0$, where $b_Ap\mathcal{B}(\mathbb{R}^d \times \{0,1\})$ denotes the positive Borel measurable functions, bounded by the constant $A \geq 0$.

Note that, for this $\psi$ the evolution equation of $v_t^\epsilon$ for on/off branching Brownian motion reads
\begin{equation}
    \label{eq:psifooBBM}
v_t^\epsilon\phi(x,i) = \mathbb{E}_{(x,i)}\bigg[\frac{1-e^{-\epsilon\phi(\xi_t^i,i)}}{\epsilon} - \int_0^t \psi(v_{t-s}^\epsilon\phi)(\xi_s^i,i) ds\bigg],
\end{equation}
where $\xi^1_t := B_t$ and $\xi^2_t := x$.
	
\begin{prop}[Sufficient condition for Lipschitz continuity of $\psi$]
\label{prop:sclcp}
If there exists some $Q \geq 0$, such that for all $(x,i) \in \mathbb{R}^d \times \{0,1\}$ we have
\begin{equation*}
|a(x,i)|, |b(x,i)| , |\Gamma((x,i),\mathbb{R}^d \times \{0,1\})| \leq Q,
\end{equation*}
then $\psi$ is Lipschitz contionus w.r.t.\ the supremum norm on $b_Ap\mathcal{B}(\mathbb{R}^d \times \{0,1\})$, i.e.\ there exists some $C_A \geq 0$ such that $\|\psi(\hat{z}) - \psi(\tilde{z}) \|_\infty \leq C_A \| \hat{z} - \tilde{z}\|_\infty$ for all $\hat{z}, z \in b_Ap\mathcal{B}(\mathbb{R}^d \times \{0,1\})$.
\end{prop}
\begin{proof}
This is a special case of Lemma 4.1 in \cite{Dynkin1994}.
Let $(x,i) \in \mathbb{R}^d \times \{0,1\}$ and $\hat{z}, \tilde{z} \in b_Ap\mathcal{B}(\mathbb{R}^d \times \{0,1\})$.
Then,
\begin{align*}
    |\psi(\hat{z})(x,i) - \psi(\tilde{z})(x,i)| &\leq |a(x,i)||\hat{z}(x,i) - \tilde{z}(x,i)|
     + |b(x,i)|\underbrace{|\hat{z}(x,i)^2-\tilde{z}(x,i)^2|}_{\leq L_A \|\hat{z} - \tilde{z}\|_\infty} \\
    &\qquad+ \int \underbrace{|\hat{z}(y) - \tilde{z}(y)|}_{\leq \|\hat{z} - \tilde{z}\|_\infty} d\Gamma((x,i);dy) \\
    &\leq Q \|\hat{z} - \tilde{z}\|_\infty + Q L_A \|\hat{z}-\tilde{z}\|_\infty + \underbrace{\Gamma((x,i);\mathbb{R}^d \times \{0,1\})}_{\leq Q} \|\hat{z}-\tilde{z}\|_\infty \\
    &\leq \underbrace{Q(2+L_A)}_{=: C_A}\|\hat{z}-\tilde{z}\|_\infty,
    \end{align*}
where we have used the local Lipschitz continuity of $z \mapsto z^2$. Since $(x,i)$ was arbitrary we get in total that $\|\psi(\hat{z}) - \psi(\tilde{z})\|_\infty \leq C_A\|\hat{z}-\tilde{z}\|_\infty$.
\end{proof}
	
	\subsubsection{Backwards in time reformulation}
	\label{ssn:ratimp}
	The idea for the construction of a solution to our evolution equation will be to work backwards in time over small time intervals. Hence it will be useful to reformulate our evolution equation appropriately. Indeed, 
we have that if
\begin{equation*}
v_t^\epsilon\phi(x,i) = \mathbb{E}_{(x,i)}\big[f(\xi_t^i,i) - \int_0^t \psi(v_{t-s}^\epsilon\phi)(\xi_s^i,i) ds\big],
\end{equation*}
then
\begin{equation*}
v_{t-r}^\epsilon\phi(x,i) = \mathbb{E}_{r,(x,i)}\big[f(\xi_t^i,i) - \int_r^t \psi(v_{t-s}^\epsilon\phi)(\xi_s^i,i) ds\big],
\end{equation*}
for $r \in [0,t]$. This follows using integration by substitution and the definition of time inhomogeneous transition probabilities.

	\subsubsection{Lipschitz continuity trick and uniform Cauchy sequence property}
	\label{ssn:lctaucsp}
	In the proof of the non-local branching solution theorem, we will exploit that the sequence of approximating solutions $(v_t^\epsilon\phi)_{\epsilon > 0}$ is a Cauchy sequence w.r.t.\ the following norm on small time intervals.
\begin{defn}[$\smalltimeinterval$-norm]
Let $\smalltimeinterval$ be a small time interval and $z: \smalltimeinterval \times \mathbb{R}^d \times \{0,1\} \to \mathbb{R}$. We define
\begin{equation*}
\|z\|_\smalltimeinterval := \sup_{s \in \smalltimeinterval} \|z(s, \cdot)\|_\infty,
\end{equation*}
where $\|z(s, \cdot)\|_\infty := \sup_{(x,i) \in \mathbb{R}^d \times \{0,1\}} |z(s,(x,i))|$.
\end{defn}
	The induction for the proof of the non-local branching solution theorem will rely on the following two results. The Lipschitz continuity trick allows to derive $\frac{1}{2}$-Lipschitz continuity of an expectation operator from the Lipschitz continuity of the branching mechanism on bounded functions.
\begin{lemma}[Lipschitz continuity trick]
    \label{lem:lct}
Set $\smalltimeinterval:=[t_0,t]$. Define for given $\psi: b_Ap\mathcal{B}(\mathbb{R}^d \times \{0,1\}) \to b_Ap\mathcal{B}(\mathbb{R}^d \times \{0,1\})$ the operator $\Psi: b\mathcal{B}(\smalltimeinterval \times \mathbb{R}^d \times \{0,1\}) \to b\mathcal{B}(\smalltimeinterval \times \mathbb{R}^d \times \{0,1\})$ via
\begin{equation*}
\Psi(z)(r,(x,i)) := \mathbb{E}_{r,(x,i)}\Big[\int_r^t\psi(z^s)(\xi^i_s,i)\, ds\Big],
\end{equation*}
where $z^s:=z(s,\cdot)$. If $\psi$ is Lipschitz continuous, i.e.\ $\|\psi(\hat{z}) - \psi(\tilde{z})\|_\infty \leq C_A\|\hat{z}-\tilde{z}\|_\infty$ for all $\hat{z}, \tilde{z} \in b_Ap\mathcal{B}(\mathbb{R}^d \times \{0,1\})$, then for $|\smalltimeinterval| := |t-t_0|$ small enough, we have
\begin{equation*}
\|\Psi(z) - \Psi(\tilde{z})\|_\smalltimeinterval \leq \frac{1}{2}\|z -\tilde{z}\|_\smalltimeinterval \quad \text{\;for all\;} z, \tilde{z} \in b_Ap\mathcal{B}(\smalltimeinterval \times \mathbb{R}^d \times \{0,1\}).
\end{equation*}
\end{lemma}
\begin{proof}
We have
\begin{align*}
    \|\Psi(z) - \Psi(\tilde{z})\|_\smalltimeinterval &\leq \sup_{r \in \smalltimeinterval} \sup_{(x,i)} \mathbb{E}_{r,(x,i)}\Big[\int_r^t \underbrace{|\psi(z^s)(\xi^i_s,i) - \psi(\tilde{z}^s)(\xi^i_s,i)|}_{\leq \|\psi(z^s) - \psi(\tilde{z}^s)\|_\infty \leq C_A \|z^s - \tilde{z}^s\|_\infty} ds\Big] \\
    &\leq C_A \sup_{r \in \smalltimeinterval} \sup_{(x,i)} \mathbb{E}_{r,(x,i)}\Big[\int_r^t \|z^s - \tilde{z}^s\|_\infty \, ds\Big] \\
    &\leq C_A \sup_{r \in \smalltimeinterval} \sup_{(x,i)} \sup_{s \in [r,t]} \|z^s - \tilde{z}^s\|_\infty |t-r| \\
    &\leq C_A |t-t_0| \sup_{r \in \smalltimeinterval} \sup_{(x,i)} \sup_{s \in [r,t]} \|z^s - \tilde{z}^s\|_\infty \\
    &= C_A |t-t_0| \sup_{r \in \smalltimeinterval} \sup_{s \in [r,t]} \|z^s - \tilde{z}^s\|_\infty \\
    &\leq \underbrace{C_A |t-t_0|}_{\leq \frac{1}{2}}  \underbrace{\sup_{s \in \smalltimeinterval} \|z^s - \tilde{z}^s\|_\infty}_{=\|z-\tilde{z}\|_\smalltimeinterval}, \\
    \end{align*}
where we have used that the expression is constant in $(x,i)$ for the equality in the penultimate step, and that $[r,t] \subset [t_0,t] = \smalltimeinterval$ in the last line, eliminating the dependence on $r$. Choosing $t$ and $t_0$ close enough yields the result.
\end{proof}
	The uniform Cauchy sequence property allows to derive the Cauchy sequence property of an expectation operator w.r.t.\ the $\Delta$-norm from the Cauchy sequence property w.r.t.\ supremum norm of a function at the righthand endpoint of each partition interval. 
\begin{lemma}(Uniform Cauchy sequence property of $F_\epsilon$)
\label{lem:uniform_cs_property_of_F}
For $f_\epsilon: \mathbb{R}^d \times \{0,1\} \to \mathbb{R}$ given, define
\begin{equation*}
F_\epsilon(r,(x,i)) := \mathbb{E}_{r,(x,i)}[f_\epsilon(\xi^i_t,i)], \quad (r,x,i) \in \smalltimeinterval \times \mathbb{R}^d \times \{0,1\}.
\end{equation*}
If $(f_\epsilon)_{\epsilon > 0}$ is a Cauchy sequence w.r.t.\ $\|\cdot\|_\infty$, i.e.\ the supremum norm on $\mathbb{R}^d \times \{0,1\}$, then $(F_\epsilon)_{\epsilon > 0}$ is a Cauchy sequence w.r.t.\ $\|\cdot\|_\smalltimeinterval$.
\end{lemma}
\begin{proof}
Let $\beta, \gamma > 0$, then we have
\begin{equation*}
    \|F_\beta - F_\gamma\|_\smalltimeinterval \leq \sup_{s \in \smalltimeinterval} \sup_{(x,i)} \mathbb{E}_{s,(x,i)}[\underbrace{|f_\beta(\xi^i_t, i) - f_\gamma(\xi^i_t,i)|}_{\leq \|f_\beta - f_\gamma\|_\infty}] = \|f_\beta - f_\gamma\|_\infty \overset{\beta,\gamma \to 0}{\longrightarrow} 0.
\end{equation*} 
\end{proof}
	To be allowed to apply the uniform Cauchy sequence property in the induction step of the non-local branching solution theorem, we also need a lemma to derive uniform convergence at a fixed point in time from uniform convergence on a small time interval.
\begin{lemma}[Supremum-norm $\smalltimeinterval$-norm lemma]
\label{lem:E_norm_Delta_norm_Lemma}
Let $f_n,f: \smalltimeinterval \times \mathbb{R}^d \times \{0,1\} \to \mathbb{R}$, $n \in \mathbb{N}$ with $f_n \overset{n \to \infty}{\longrightarrow} f$ w.r.t.\ $\|\cdot\|_\smalltimeinterval$, then for $r \in \smalltimeinterval$ fixed, we have that $f_n(r, \cdot) \overset{n \to \infty}{\longrightarrow} f(r,\cdot)$ w.r.t.\ $\|\cdot\|_\smalltimeinterval$ and $\|\cdot\|_\infty$.
\end{lemma}
\begin{proof}
The proof is standard and therefore being omitted.
\end{proof}
	
	\subsubsection{Various boundedness properties}
	\label{ssn:vbp}
	The goal of this subsection is to show the boundedness of the approximating solution candidates $v_t^\epsilon\phi$, see Proposition \ref{prop:bddvtepsphi}. 
	This is necessary
	for the proof of the non-local branching solution theorem, because the applicability of the Lipschitz continuity trick relies on this boundedness. 
	
	The following result prepares the proof of Proposition \ref{prop:bddvtepsphi} by exploiting some boundedness from below of the branching mechanism (Proposition \ref{prop:bddv}). Then we have to check whether the branching mechanism of on/off branching Brownian motion satisfies this boundedness from below (Proposition \ref{prop:bddvte}), and 
	finally 
	we have to show  boundedness for the initial step of the induction (Lemma \ref{lem:bddfe}).
\begin{lemma}[Boundedness of $v_\epsilon^t$]
\label{prop:bddv}
Let $\Delta = [t_0,t]$ and suppose that there exists some $Q \geq 0$, such that
\begin{equation}
\label{eq:bovte1}
\psi(z) \geq -Q\|z\|_\infty \quad \text{\;for all\;} z \in bp\mathcal{B}(\mathbb{R}^d \times \{0,1\}).
\end{equation}
Furthermore suppose that $|\smalltimeinterval|Q \leq \frac{1}{2}$, where $|\smalltimeinterval|$ denotes the length of the interval $\smalltimeinterval$. If
\begin{equation}
\label{eq:bovte0}
v_\epsilon^t = F_\epsilon - \Psi(v_\epsilon^t),
\end{equation}
where
\begin{align}
    \label{eq:bddvvteps}
    v_\epsilon^t(r,(x,i)) &:= v_{t-r}^\epsilon\phi(x,i), \\
    \label{eq:bddvPsi}
\Psi(v_\epsilon^t)(r,(x,i)) &:= \mathbb{E}_{r,(x,i)} \Big[\int_r^t \psi(v_{t-s}^\epsilon\phi)(\xi_s^i,i) ds\Big], \\
\label{eq:bddvF}
F_\epsilon(r,(x,i)) &:= \mathbb{E}_{r,(x,i)}[f_\epsilon(\xi_t^i,i)],
\end{align}
for some suitable $f_\epsilon$, then we have that
\begin{equation*}
\|v_\epsilon^t\|_\smalltimeinterval \leq 2 \|F_\epsilon\|_\smalltimeinterval.
\end{equation*}
\end{lemma}
This corresponds to
\cite{Dynkin1994} Equation 4.28 and below Equation 4.32 and Equation 4.33 and p.\ 58. 
\begin{proof}
Using \eqref{eq:bovte0} and \eqref{eq:bovte1}, we get for $r \in \smalltimeinterval$ that
\begin{equation*}
    v_\epsilon^t(r, (x,i)) \leq \|F_\epsilon\|_\smalltimeinterval + \mathbb{E}_{r,(x,i)}\Big[\int_r^t -\psi(v_{t-s}^\epsilon\phi)(\xi_s^i,i) \, ds\Big] \leq \|F_\epsilon\|_\smalltimeinterval + {(t-r)Q}
    \|v\|_\smalltimeinterval.
\end{equation*}
Since this upper bound does not depend on $r$ and $(x,i)$, and $(t-r)Q \le \nicefrac 12$, we have that
\begin{equation*}
\|v_\epsilon^t\|_\smalltimeinterval \leq \|F_\epsilon\|_\smalltimeinterval + \frac{1}{2}\|v_\epsilon^t\|_\smalltimeinterval.
\end{equation*}
Subtracting $\frac{1}{2}\|v_\epsilon^t\|_\smalltimeinterval$ on both sides gives the result.
\end{proof}

The branching mechanism of on/off branching Brownian motion immediately fulfills the boundedness from below condition.
\begin{prop}
    \label{prop:bddvte}
For $\psi$ as in Equation \eqref{eq:sfopsi1} with $a(x,i),b(x,i) \geq 0$ and $\Gamma((x,i), \mathbb{R}^d \times \{0,1\}) \leq Q$ for all $(x,i) \in \mathbb{R}^d \times \{0,1\}$, we have
\begin{equation*}
\psi(z) \geq -Q\|z\|_\infty \quad \text{\;for all\;} z \in bp\mathcal{B}(\mathbb{R}^d \times \{0,1\}).
\end{equation*}
\end{prop}
We also have to check the following boundedness for the inital step of the induction in the proof of Proposition \ref{prop:bddvtepsphi}.
\begin{lemma}[Boundedness of $f_\epsilon$]
\label{lem:bddfe}
We have for $\phi \in bp\mathcal{B}(\mathbb{R}^d \times \{0,1\})$ and $\epsilon > 0$, that
\begin{equation*}
f_\epsilon(x,i) := \frac{1- e^{-\epsilon\phi(x,i)}}{\epsilon} \leq \phi(x,i) \leq \|\phi\|_\infty.
\end{equation*}
\end{lemma}
\begin{proof}
This follows from a first order Taylor expansion of $z \mapsto \frac{1 - e^{-\epsilon z}}{\epsilon}, z \in \mathbb{R}$.
\end{proof}

Finally we can prove the main proposition of this subsection.
\begin{prop}[Boundedness of $v_t^\epsilon\phi$]
    \label{prop:bddvtepsphi}
Let $\phi \in bp\mathcal{B}(\mathbb{R}^d \times \{0,1\})$ and $t \geq 0$ fixed. If
\begin{equation*}
v_{t-r}^\epsilon\phi(x,i) = \mathbb{E}_{r,(x,i)}\bigg[\frac{1-e^{-\epsilon \phi (\xi_t^i,i)}}{\epsilon} - \int_r^t \psi(v^\epsilon_{t-s}\phi)(\xi_s^i,i) ds\bigg],
\end{equation*}
and 
\begin{equation*}
\psi(z) \geq -Q\|z\|_\infty \quad \text{\;for all\;} z \in bp\mathcal{B}(\mathbb{R}^d \times \{0,1\}),
\end{equation*}
then there exists an $n \in \mathbb{N}$ such that for any $\epsilon > 0$ we have that
\begin{equation*}
\| v_t^\epsilon\phi \|_\infty \leq 2^n\|\phi\|_\infty =: A.
\end{equation*}
\end{prop}
\begin{proof}
The proof goes by induction. Let $0=t_0 < t_1 < ... < t_n = t$ be a partition with mesh fine enough. (Note that the $n$ in the upper bound from the statement is the number of subintervals of this partition.)
We will show that
\begin{equation}
\label{eq:bddvte_star}
\|v_\epsilon^t\|_{[t_{j-1}, t_j)} \leq 2^{n-j+1}\|\phi\|_\infty \leq 2^{n+1}\|\phi\|_\infty,
\end{equation}
for any $j \in \{n,...,1\}$.

According to the localization Lemma \ref{lem:ll}, we have
\begin{equation}
    \label{eq:fromlocalizationlemma1}
v_{t-t}^\epsilon\phi(x,i) = v_0^\epsilon\phi(x,i) = f_\epsilon(x,i)
\end{equation}
and
\begin{equation}
    \label{eq:fromlocalizationlemma2}
v_{t-r}^\epsilon\phi(x,i) = \mathbb{E}_{r,(x,i)}\Big[v_{t-t_j}^\epsilon\phi(\xi_{t_j}^i,i) - \int_r^{t_j} \psi(v_{t-s}^\epsilon\phi)(\xi_s^i,i) ds\Big],
\end{equation}
for $j=1,...,n$ and $r \in [t_{j-1}, t_j)$.

\underline{Initial step:}
Let $j=n$, then we have $t_j = t_n= t$ and $t_{j-1} = t_{n-1}$, i.e.\ $\Delta = [t_{n-1}, t_n]$.
Recall the definitions of $v^t_\epsilon$, $\Psi(v^t_\epsilon)$, $F_\epsilon$ from \eqref{eq:bddvvteps}, \eqref{eq:bddvPsi}, \eqref{eq:bddvF} with $f_\epsilon(x,i) = \frac{1-e^{-\epsilon\phi(x,i)}}{\epsilon}$.
Then, we have by \eqref{eq:fromlocalizationlemma1} and \eqref{eq:fromlocalizationlemma2} that
\begin{equation}
\label{eq:bddvte1b}
v^t_\epsilon = F_\epsilon - \Psi(v^t_\epsilon).
\end{equation}
From Lemma \ref{prop:bddv} we obtain
\begin{equation*}
\|v^t_\epsilon\|_{[t_{n-1}, t_n]} \leq 2\|F_\epsilon\|_{[t_{n-1}, t_n]}.
\end{equation*}
By the boundedness of $f_\epsilon$ (Lemma \ref{lem:bddfe}), we get for $r \in [t_{n-1}, t_n]$ and $(x,i) \in \mathbb{R}^d \times \{0,1\}$, that
\begin{equation*}
    F_\epsilon(r,(x,i)) = \mathbb{E}_{r,(x,i)}[f_\epsilon(\xi_t^i,i)] \leq \|\phi\|_\infty
\end{equation*}
and therefore
\begin{equation*}
\|v^t_\epsilon\|_{[t_{n-1}, t_n]} \leq 2\|F_\epsilon\|_{[t_{n-1}, t_n]} \leq 2\|\phi\|_\infty = 2^{n-n+1}\|\phi\|_\infty = 2^{n-j+1}\|\phi\|_\infty.
\end{equation*}

\underline{Induction hypothesis:} Suppose for some $j \in \{n-1,...,1\}$ it holds that
\begin{equation*}
\|v^t_\epsilon\|_{[t_j, t_{j+1})} \leq 2^{n-j+1}\|\phi\|_\infty.
\end{equation*}

\underline{Induction step ($j \implies j-1$):} We have to show that $\|v_\epsilon^t\|_{[t_{j-1}, t_j)} \leq 2^{n-(j-1)+1} \|\phi\|_\infty = 2^{n-j+2} \|\phi\|_\infty$.
Define
\begin{align}
    \nonumber F^{(j)}_\epsilon(r,(x,i)) &:= \mathbb{E}_{r,(x,i)}[v^\epsilon_{t-t_j}\phi(\xi_{t_j}^i,i)],\\
    \label{eq:bddvte2a}\Psi^{(j)}(v^t_\epsilon)(r,(x,i)) &:= \mathbb{E}_{r, (x,i)}\Big[\int_r^{t_j} \psi(v_{t-s}^\epsilon\phi)(\xi_s^i, i) ds\Big],
\end{align}  
then
\begin{equation}
\label{eq:bddvte2b}
v^t_\epsilon = F^{(j)}_\epsilon - \Psi^{(j)}(v^t_\epsilon).
\end{equation}
We have from  \eqref{eq:bddvte2a} and \eqref{eq:bddvte2b}, and choosing $f_\epsilon(x,i) := v_{t-t_j}^\epsilon\phi(x,i)$ and $t:=t_j$ in Proposition \ref{prop:bddv}, that
\begin{equation}
	\label{eq:Prop2.11}
\|v_\epsilon^t\|_{[t_{j-1}, t_j)} \leq 2 \|F^{(j)}_\epsilon\|_{[t_{j-1}, t_j)}.
\end{equation}
Now for $r \in [t_{j-1}, t_j)$ and $(x,i) \in \mathbb{R}^d \times \{0,1\}$, we get
\begin{equation*}
F^{(j)}_\epsilon(r,(x,i)) = \mathbb{E}_{r,(x,i)}[v^\epsilon_{t-t_j}\phi(\xi_{t_j}^i,i)] \leq 2^{n-j+1}\|\phi\|_\infty,
\end{equation*} 
because
\begin{equation*}
    v^\epsilon_{t-t_j}\phi(\xi_{t_j}^i,i) \leq \sup_{s \in [t_j, t_{j+1})}\sup_{(x,i) \in \mathbb{R}^d \times \{0,1\}} |v^t_\epsilon(s,(x,i))| \leq 2^{n-j+1} \|\phi\|_\infty,
\end{equation*}
by the induction hypothesis.
And therefore we also get
\begin{equation*}
\|F_\epsilon^{(j)}\|_{[t_{j-1}, t_j)} = \sup_{r \in [t_{j-1}, t_j)} \sup_{(x,i) \in \mathbb{R}^d \times \{0,1\}} |\underbrace{F^{(j)}_\epsilon(r,(x,i))}_{\leq 2^{n-j+1} \|\phi\|_\infty}| \leq 2^{n-j+1} \|\phi\|_\infty.
\end{equation*}
Now, \eqref{eq:Prop2.11} concludes the induction.

To finish the proof of the theorem, we obtain from Equation \eqref{eq:bddvte_star} with $j := 1$ that
\begin{equation*}
\|v_{t-0}^\epsilon\phi\|_\infty \leq \sup_{s \in [0,t_1)} \sup_{(x,i) \in \mathbb{R}^d \times \{0,1\}} |v_\epsilon^t(s,(x,i))| \leq 2^n\|\phi\|_\infty.
\end{equation*}   
\end{proof}

Lastly, we derive a quick corollary that we will need later for the application of the dominated convergence theorem to derive the limiting evolution equation on small time intervals. It follows immediately from \eqref{eq:bddvte_star}.
\begin{cor}[Boundedness of $v_{t-s}^\epsilon\phi$]
We have
\begin{equation*}
\|v_{t-s}^\epsilon\phi \|_\infty \leq 2^{n+1}\|\phi\|_\infty, 
\end{equation*}
for any $s \in [0,t)$.
\end{cor}

	\subsubsection{Non-local branching solution theorem}
	\label{ssn:nlbst}
	Now we are ready to state and prove the main theorem of this section.
\begin{theorem}[Non-local branching solution theorem]
\label{thm:nlbst}
Let $\psi$ be Lipschitz continuous w.r.t.\ the supremum norm on $b_Ap\mathcal{B}(\mathbb{R}^d \times \{0,1\})$ for some $A \geq 0$ and let
\begin{equation}
    \label{eq:nlbst0}
v_{t-r}^\epsilon\phi(x,i) = \mathbb{E}_{r,(x,i)}\bigg[\frac{1-e^{-\epsilon\phi(\xi_t^i,i)}}{\epsilon} - \int_r^t \psi(v_{t-s}^\epsilon\phi)(\xi_s^i,i) ds\bigg],
\end{equation}
for $r \in [0,t]$. Furthermore let $0=t_0 < t_1 < ... < t_n = t$ be a partition of $[0,t]$ with mesh fine enough.

Then, for fixed $t \geq 0$, we have that $(v_t^\epsilon\phi)_{\epsilon > 0}$ is a Cauchy sequence w.r.t.\ $\|\cdot\|_{[t_{j-1}, t_j)}$ for every $j=1,...,n$ and is therefore convergent. The limit $v$ solves
\begin{equation*}
v_t\phi(x,i) = \mathbb{E}_{(x,i)}\Big[ \phi(\xi_t^i,i) - \int_0^t \psi(v_{t-s}\phi)(\xi_s^i,i)ds\Big].
\end{equation*}
\end{theorem}
The proof is a specialization of the proof of Theorem 4.2 in \cite{Dynkin1994}.
\begin{proof}
Set $f_\epsilon(x,i):= \frac{1-e^{-\epsilon\phi(x,i)}}{\epsilon}$.
Recall Equations \eqref{eq:fromlocalizationlemma1} and \eqref{eq:fromlocalizationlemma2} from the application of the localization lemma.
The proof goes (again) by backwards induction over $j$.

\underline{Initial step:}
Let $\Delta := [t_{n-1}, t_n]$.
Recall Equation \eqref{eq:bddvte1b}.
Let $\gamma, \beta > 0$, then by the boundedness of $v_t^\epsilon\phi$, that is Proposition \ref{prop:bddvtepsphi}, and the Lipschitz continuity trick, we get
\begin{equation*}
    \|v^t_\beta - v^t_\gamma\|_\Delta \leq \|F_\beta - F_\gamma\|_\Delta + \|\Psi(v^t_\beta) - \Psi(v^t_\gamma)\|_\Delta \leq \|F_\beta - F_\gamma\|_\Delta + \frac{1}{2} \|v^t_\beta - v^t_\gamma\|_\Delta.
\end{equation*}    
From this, as before, we obtain
\begin{equation*}
\|v^t_\beta - v^t_\gamma\|_\Delta \leq 2\|F_\beta - F_\gamma\|_\Delta \overset{\gamma, \beta \to 0}{\longrightarrow} 0,
\end{equation*}
where we have used that the uniform convergence of $f_\epsilon$ implies the uniform Cauchy sequence property of $F_\epsilon$ by Lemma \ref{lem:uniform_cs_property_of_F}.
In particular we have that $(v^t_\epsilon)_{\epsilon > 0}$ is a Cauchy sequence w.r.t.\ $\|\cdot\|_\Delta$, meaning that it converges uniformly on $\Delta \times \mathbb{R}^d \times \{0,1\}$ to some limit $v$. 
Now sending $\epsilon \to 0$, we get using the dominated convergence theorem that
\begin{equation*}
    v_{t-r}\phi(x,i) = \mathbb{E}_{r,(x,i)}\Big[ \phi(\xi_t^i,i) - \int_r^t \psi(v_{t-s}\phi)(\xi_s^i,i)\, ds\Big],
\end{equation*}
where we used boundedness of $v_{t-s}^\epsilon\phi$ and the Lipschitz continuity of $\psi$.

\underline{Induction hypothesis:} Suppose, for some $j \in \{n-1,...,1\}$, that we have already constructed a solution $v: [t_j, t_{j+1}) \times \mathbb{R}^d \times \{0,1\} \to \mathbb{R}$.

\underline{Induction step:}
Recall Equation \eqref{eq:bddvte2b}. Let $\Delta:=[t_{j-1},t_j)$ and $\beta, \gamma > 0$, then we get
\begin{equation*}
\|v^t_\beta - v^t_\gamma\|_\Delta \leq \|F^{(j)}_\beta - F^{(j)}_\gamma\|_\Delta + \|\Psi^{(j)}(v^t_\beta) - \Psi^{(j)}(v^t_\gamma)\|_\Delta
\end{equation*}
and analogously
\begin{equation}
    \label{eq:nlbst3}
\|v^t_\beta - v^t_\gamma\|_\Delta \leq 2\|F^{(j)}_\beta - F^{(j)}_\gamma\|_\Delta \overset{\gamma, \beta \to 0}{\longrightarrow} 0.
\end{equation}
We get that this is going to zero by using Lemma \ref{lem:uniform_cs_property_of_F}, with $f_\epsilon(x,i) := v^\epsilon_{t-t_j}\phi(x,i)$. Note that by the induction hypothesis, we obtain the convergence of $f_\epsilon$ w.r.t.\ $\|\cdot\|_\infty$ from the convergence of $v_\epsilon^t$ on $[t_j, t_{j+1}) =: \tilde{\Delta}$ w.r.t.\ $\|\cdot\|_{\tilde{\Delta}}$ by Lemma \ref{lem:E_norm_Delta_norm_Lemma}. By Equation \eqref{eq:nlbst3} we have that $v_\epsilon^t$ is also a Cauchy sequence w.r.t.\ $\|\cdot\|_\Delta$. Therefore $v_{t-s}^\epsilon\phi$ converges for any $s \in \Delta$ w.r.t.\ $\|\cdot\|_\infty$. We call its limit also $v_{t-s}\phi$.
Now sending $\epsilon \to 0$ we get using the dominated convergence theorem and the uniform convergence of $v_{t-t_j}^\epsilon\phi$ to $v_{t-t_j}\phi$, that for $r \in \Delta$ it holds
\begin{align*}
    v_{t-r}\phi(x,i) = \mathbb{E}_{r,(x,i)}[ v_{t-t_j}\phi(\xi_{t_j}^i,i) - \int_r^{t_j} \psi(v_{t-s}\phi)(\xi_s^i,i)ds].
\end{align*}
Now that the induction is done, we have shown that this equality holds for all $j=n,...,1$. Since we have, by choosing $r:= t \in [t_{n-1}, t]$, that
\begin{equation*}
    v_{t-t}\phi(x,i) = \phi(x,i),
\end{equation*}
we can apply the other direction of the localization lemma to obtain
\begin{equation*}
v_{t-r}\phi(x,i) = \mathbb{E}_{r,(x,i)}[\phi(\xi_t^i,i) - \int_r^t \psi(v_{t-s}\phi)(\xi_s^i,i) ds], \quad \text{\;for all\;} r \in [0,t].
\end{equation*}
In particular for $r=0$.
\end{proof}

\begin{cor}[Uniform convergence of $v_t^\epsilon\phi$ to $v_t\phi$]
\label{cor:nlbst}
It holds that
\begin{equation*}
\| v_t\phi - v_t^\epsilon\phi \|_\infty \overset{\epsilon \to 0}{\longrightarrow} 0.
\end{equation*}
\end{cor}
\begin{proof}
The claim follows
by the previous theorem and Lemma \ref{lem:E_norm_Delta_norm_Lemma}.
\end{proof}

	\subsubsection{Properties of solutions}
	\label{ssn:pos}
	We briefly collect boundedness, uniqueness and the  semi-group property of the limiting solution.
\begin{prop}[Boundedness of the solution $v$]
If
\begin{equation*}
v_t\phi(x,i) = \mathbb{E}_{(x,i)}[\phi(\xi_t^i,i) - \int_0^t \psi(v_{t-s}\phi)(\xi_s^i,i) ds],
\end{equation*}
then there exists an $A \geq 0$ such that
\begin{equation*}
\|v_t\phi\|_\infty \leq A.
\end{equation*}
\end{prop}
Note that this is the same upper bound as for $v_t^\epsilon\phi$.
\begin{proof}
This follows by Proposition \ref{prop:bddv} and backwards induction.
\end{proof}
\begin{prop}[Uniqueness of solutions to evolution equations]
The solution from the non-local branching solution theorem is unique.
\end{prop}
\begin{proof}
It is enough to consider uniqueness only on each small time interval $\Delta = [t_{j-1}, t_j)$, because the global solution (on the whole time axis, glued together using the localization lemma) is then assembled from unique pieces and therefore the global solution is unique as well. The proof also proceeds as a backwards induction.
\end{proof}
	
\begin{prop}[Semigroup property of evolution equations]
Let $(V_t)_{t \geq 0}$ be the solution of
\begin{equation}
\label{eq:spomtsee1}
V_t\phi(x,i) = \mathbb{E}_{(x,i)}\Big[\phi(\xi_t^i,i) - \int_0^t \psi(V_{t-r}\phi)(\xi_r^i,i)\, dr\Big],
\end{equation}
then $(V_t)_{t \geq 0}$ is a semigroup, i.e.\ $V_s \circ V_t = V_{s+t}$ and $V_0=$Id.
\end{prop}
\begin{proof}
The proof is standard and therefore being omitted.
\end{proof}
	
	\subsection{Convergence of the finite dimensional distributions and existence and duality for on/off super-Brownian motion}
	\label{ssn:fdd}
	In this section our goal is to establish existence and duality for on/off super-Brownian motion (Theorem \ref{thm:ooSBMExistencewithpsi}). We begin by showing tightness of the one-dimensional marginals of the rescaled empirical on/off BBMs by considering their moment measures (Proposition \ref{prop:fmmoobbm}). Then we show weak convergence by identification of a unique limit point using functional Laplace transforms (Lemma \ref{lem:wco1mosemooSBM}). By showing that the functional Laplace transforms of this limit point generate a Markov process (and satisfy the branching property), we arrive at the on/off super-Brownian motion existence Theorem \ref{thm:ooSBMExistencewithpsi}, which also provides a deterministic dual.
\begin{prop}[First moment measure of on/off BBM]
    \label{prop:fmmoobbm}
For fixed $t \geq 0$, let $(P_t^\epsilon)_{\epsilon > 0}$ be the distributions of the rescaled empirical measures of on/off BBM, i.e.\ $P_t^\epsilon := Law(\Zooeps)$, started in a Poisson point process with intensity $\frac{\mu}{\epsilon}$.
Then their first moment measures all coincide and can be represented by a linear form on $bp\mathcal{B}(\mathbb{R}^d \times \{0,1\})$ given by
\begin{equation*}
M_{P_t^\epsilon}(\phi) = \langle \mu, \ooBMtrans_t \phi\rangle, \quad \phi \in bp\mathcal{B}(\mathbb{R}^d \times \{0,1\}),
\end{equation*}
where $(\ooBMtrans_t)_{t \geq 0}$ is the transition semigroup of on/off Brownian motion.
In particular this represents a finite measure, and $(P_t^\epsilon)_{\epsilon > 0}$ is a tight set of probability measures.
\end{prop}
\begin{proof}
We have
\begin{equation}
    M_{P_t^\epsilon}(\phi) = \epsilon \mathbb{E}[\langle \Zoo_t, \phi \rangle] = \epsilon \mathbb{E}\left[\sum_{i=1}^{n(0)} \mathbb{E}\big[\langle \Zoo_t, \phi \rangle \big| \Zoo_0 = \delta_{(X_i(0), \adstate_i(0))}\big]\right],
    \end{equation}
where we have used the branching principle for the scaled empirical measure in the last equality.
Note that
\begin{align*}
    \mathbb{E}\big[\langle \Zoo_t, \phi \rangle \big| \Zoo_0 = \delta_{(X_i(0),\adstate_i(0))}\big] &= \mathbb{E}\Big[ \; \sum_{k=0}^\infty 1_{\{n(t) = k\}} \sum_{j=1}^k \phi(X_j(t), \adstate_j(t)) \; \Big|\Zoo_0 = \delta_{(X_i(0),\adstate_i(0))}\Big] \\
    &= \underbrace{\mathbb{E}[n(t)]}_{=1}\ooBMtrans_t\phi(X_i(0), \adstate_i(0)),
\end{align*}    
where we have used that we conditioned on the particle system starting in one particle and the criticality of on/off branching Brownian motion, meaning that the expected number of offspring is constant and equal to one. 
From this we get
\begin{equation*}
    \epsilon \mathbb{E}\left[\sum_{i=1}^{n(0)} \mathbb{E}[\langle \Zoo_t, \phi \rangle | \Zoo_0 = \delta_{X_i(0),\adstate_i(0)}]\right] = \epsilon \mathbb{E}\left[\int \ooBMtrans_t\phi d\Zoo_0 \right] = \epsilon \int \ooBMtrans_t\phi d\frac{\mu}{\epsilon} = \langle \mu, \ooBMtrans_t\phi \rangle,
\end{equation*}    
where in the penultimate equality we have used the mean value property of Poisson point processes and that the initial particle configuration has intensity measure $\frac{\mu}{\epsilon}$. In particular this expression is independent of $\epsilon$, which means that all the moment measures of $(P_t^\epsilon)_{\epsilon > 0}$ coincide.

Next we check if the moment measure is actually a finite measure. For this we note that
\begin{align*}
    M_{P_t^\epsilon}(1_{\mathbb{R}^d \times \{0,1\}}) = \langle \mu, \ooBMtrans_t 1_{\mathbb{R}^d \times \{0,1\}} \rangle \leq \|\ooBMtrans_t\| \mu(\mathbb{R}^d \times \{0,1\}) < \infty,
\end{align*}    
where we have used that $\mu$ is a finite measure and $\ooBMtrans_t$, as a transition operator, is bounded.
Finally, since the moment measures of $(P_t^\epsilon)_{\epsilon > 0}$  are all identical, they are trivially tight, and hence also the marginal distributions  $(P_t^\epsilon)_{\epsilon > 0}$ form a tight family.
\end{proof}

	Next we prepare the weak convergence of the one-dimensional marginals of the scaled empirical measure with the following two propositions.
We begin with an explicit formula for the functional Laplace transform of the rescaled process using the approximating solution to the evolution equation.
\begin{prop}[Functional Laplace transform of scaled empirical measure]
    \label{prop:fltsem}
For all $t \ge 0$ we have that
\begin{equation*}
\mathbb{E}\big[e^{-\langle \Zooeps, \phi \rangle}\big] = \exp\Big(-\int_{\mathbb{R}^d \times \{0,1\}} v_t^\epsilon\phi(x,i) \; d\mu(x,i)\Big),
\end{equation*}
where $\mathbb{E}[\cdot]$ is the expectation under which the process starts randomly in a Poisson point measure with intensity measure $\frac{\mu}{\epsilon}$. 
\end{prop}
\begin{proof}
The branching property for the empirical measures implies that
\begin{equation*}
\Zooeps \overset{d}{=} \sum_{j=1}^{n(0)} \Zooeps(X_j(0), \adstate_j(0)),
\end{equation*}
where $\Zooeps(x,i)$ is an independent copy of $\Zooeps$, started with a single particle at $(x,i)$. From this we get
\begin{align*}
    \mathbb{E}[e^{-\langle \Zooeps, \phi \rangle}] &= \mathbb{E}\left[\mathbb{E}[ \; \prod_{j=1}^{n(0)} e^{-\langle \Zooeps(X_j(0), \adstate_j(0)), \phi \rangle} \; |(X_i(0), \adstate_i(0)), i=1,...,n(0)]\right] \\
    &= \mathbb{E}\left[\prod_{j=1}^{n(0)}w_t^\epsilon\phi(X_j(0), \adstate_j(0))\right] \\
    &= \mathbb{E}\left[\exp\left(\sum_{j=1}^{n(0)}\ln(w_t^\epsilon\phi(X_j(0), \adstate_j(0)))\right)\right] \\
    &= \exp\left(- \int 1 - e^{\ln(w_t^\epsilon\phi(x,i))} d\frac{\mu}{\epsilon}(x,i)\right) 
    = \exp\left(- \int v_t^\epsilon\phi(x,i) d\mu(x,i)\right).
    \end{align*}
\end{proof}
Now we check the convergence of the integral of the approximating solution to the integral of the solution to the evolution equation, which implies the convergence of the corresponding Laplace transforms. 
\begin{prop}[Convergence of Laplace transform of scaled empirical measure]
    \label{prop:cltsem}
We have that
\begin{equation*}
\int_{\mathbb{R}^d \times \{0,1\}} v_t^{\epsilon_k}\phi(x,i) \; d\mu(x,i) \overset{k \to \infty}{\longrightarrow} \int_{\mathbb{R}^d \times \{0,1\}} V_t\phi(x,i) \; d\mu(x,i),
\end{equation*}
for $\epsilon_k \overset{k \to \infty}{\longrightarrow} 0$.
\end{prop}
\begin{proof}
We have
\begin{align*}
    \bigg|\int_{\mathbb{R}^d \times \{0,1\}} v_t^{\epsilon_k}\phi(x,i) \; d\mu(x,i) - \int_{\mathbb{R}^d \times \{0,1\}} V_t\phi(x,i) \; d\mu(x,i) \bigg| &\leq \mu(\mathbb{R}^d \times \{0,1\}) \|v_t^{\epsilon_k}\phi(x,i) - V_t\phi(x,i)\|_\infty,
\end{align*}
where the right-hand side converges to 0 as $k \to \infty$, by the non-local branching solution theorem.
\end{proof}
Finally we arrive at the first main goal of this section, namely the weak convergence of the one-dimensional marginals of the rescaled process.
\begin{lemma}(Weak convergence of one-dim.\ marginals of scaled empirical measure to on/off SBM).
\label{lem:wco1mosemooSBM}
For fixed $t \geq 0$, let $P_t^\epsilon := Law(\Zooeps)$, where $\Zooeps$ is started according to the intensity measure $\frac{\mu}{\epsilon}$. 
Then there exists some distribution $P_t$, such that
\begin{equation*}
P_t^\epsilon \overset{\epsilon \to 0}{\Longrightarrow} P_t =: P_{t,\mu},
\end{equation*}
i.e.\ the one-dim marginals of the scaled empirical measure converge weakly in $\mathcal{M}_1(\mathcal{M}_F(\mathbb{R}^d \times \{0,1\}))$ to some probability distribution. 
This limiting distribution has the functional Laplace transform
\begin{equation*}
L_{P_{t,\mu}}(\phi) = e^{-\int V_t \phi(x,i) d\mu(x,i)}
\end{equation*}
and is therefore the one-dimensional marginal of an on/off super-Brownian motion.
\end{lemma}
\begin{proof}
Since $(P_t^\epsilon)_{\epsilon > 0}$ is a tight set of probability measures, we get via  Prokhorov's theorem the existence of at least one limit point. Since all scaled functional Laplace transforms converge to the same limit given as above (Propositions \ref{prop:fltsem} and \ref{prop:cltsem}), and since the functional Laplace transform restricted to bounded positive continuous functions uniquely determines a probability distribution on the finite measures (of some Polish space), the result follows.
\end{proof}

Similar arguments also yield the convergence of the {\em finite-dimensional marginals} via convergence of the corresponding {\em finite-dimensional functional Laplace-transforms}. We omit the details for brevity (see e.g.\ \cite{GorostizaLopezMimbela1990}).

We are now ready to state the on/off super-Brownian motion existence theorem.
\begin{theorem}(On/off super-Brownian motion existence theorem)
\label{thm:ooSBMExistencewithpsi}
There exists an $\mathcal{M}_F(\mathbb{R}^d \times \{0,1\})$-valued branching Markov process $(\sbm_t)_{t \geq 0}$, whose transition probabilities are determined by the functional Laplace transforms
\begin{equation}
\label{eq:mtset0}
\mathbb{E}_\mu[e^{-\langle \sbm_t, \phi \rangle}] = \exp(-\int_{\mathbb{R}^d \times \{0,1\}} V_t\phi(x,i) \; d\mu(x,i)),
\end{equation}
where $(V_t)_{t \geq 0}$ solves the evolution equation
\begin{equation}
\label{eq:mtset1}
V_t\phi(x,i) = \mathbb{E}_{(x,i)}[ \phi(\xi_t^i,i) - \int_0^t \psi(V_{t-s}\phi)(\xi_s^i,i)ds].
\end{equation}
\end{theorem}

\begin{proof}
From the non-local branching solution theorem we already know that \eqref{eq:mtset1} has a unique solution.
By Lemma \ref{lem:wco1mosemooSBM} we know that the one-dimensional distributions of the empirical measures of the approximating on/off BBMs converges
to a unique limit with functional Laplace transform
\begin{equation*}
	L_{t,\mu}(\phi) = \exp(-\int_{\mathbb{R}^d \times \{0,1\}} V_t\phi(x,i) \; d\mu(x,i))
\end{equation*}
(and similar statements holds for the finite-dimensional distributions).
We will show that the family $\{L_{t, \mu}\}$ is Markov process generating and has the branching property.

Firstly, by definition, $L_{t,\mu}$ is a functional Laplace transform of some measure and we have for the initial condition that $P_{0,\mu} = \delta_{\mu}$. 
Secondly, measurability w.r.t.\ $\borelsets(\mathcal{M}_F(\mathbb{R}^d \times \{0,1\}))$ of the function $\mu \mapsto L_{t,\mu}(\phi)$ follows from continuity.
And thirdly, the Chapman-Kolmogorov equation and therefore the Markov property follows by the semigroup property of $(V_t)_{t \geq 0}$.
Finally, the branching property follows from
\begin{equation*}
L_{t, \mu_1+\mu_2}(\phi) = \exp(-\int V_t\phi(x,i) d\mu_1(x,i))\exp(-\int V_t\phi(x,i) d\mu_2(x,i)) = L_{t,\mu_1}(\phi)L_{t,\mu_2}(\phi).
\end{equation*}    
\end{proof}

	\subsection{Tightness on path-space, martingale problem and path properties of on/off SBM}
\label{sec:tightness}
Next we turn to the tightness of on/off binary branching Brownian motion on the path space. For this we use the same proof strategy as in \cite{E00}.

\begin{theorem}[Duality of branching Brownian motion]
    Let $Z$ be an on/off branching Brownian motion with binary branching, considered as a measure-valued process.
    Then
    \begin{equation*}
        \mathbb{E}_\mu [e^{\langle Z_t, \ln \phi\rangle}] = \exp(\langle \mu, \ln u(t) \rangle),
    \end{equation*}
    where the deterministic dual process $u$ solves the system of equations
    \begin{equation*}
        \begin{cases}
            \partial_t u(t, (x,1)) = \frac{1}{2}\Delta u(t,(x,1)) + \frac{\gamma}{2}\bigl( \; \Phi[u(t,(x,1))] -u(t,(x,1))\; \bigr) + c\bigl(\; u(t,(x,0)) - u(t,(x,1))\; \bigr), \\
            \partial_t u(t, (x,0)) = \tilde{c} \bigl(\;u(t, (x,1)) - u(t, (x,0))\; \bigr),
        \end{cases}
    \end{equation*}
    with $u(0, (x,i)) = \phi(x,i)$ and $\Phi[u] := \frac{1}{2} + \frac{1}{2}u^2$.
\end{theorem}
\begin{proof}
    This is the `Laplace duality', see Example 1.5 of \cite{2014_JansenKurt_OnTheNotionsOfDualityForMarkovProcesses}, analog of the `moment duality' in \cite[Theorem 1.7]{BHN22} for the binary branching case or \cite[Theorem 1.4]{E00}, since $Z$ is a sum of diracs.
\end{proof}
In what follows the superscript ``$\%$" stands for ``on/off".
\begin{prop}[Generator of on/off binary branching Brownian motion]
	The generator $\mathcal{L}_1^\%$ of on/off binary branching Brownian motion acts on
    functions of the form $f(\mu) := e^{\langle \mu, \ln \phi \rangle}$ as
    \begin{equation*}
        \mathcal{L}_1^\% f(\mu) = \langle \mu, \frac{1}{u(0)} \partial_t u(0) \rangle e^{\langle \mu, \ln u(0) \rangle}.
    \end{equation*}
\end{prop}
\begin{proof}
By definition of the generator, we have that
\begin{equation*}
    \mathcal{L}_1^\% f(\mu) = \lim_{h \to 0} \frac{1}{h}(\mathbb{E}_\mu[e^{\langle Z_h, \ln \phi \rangle}] - e^{\langle \mu, \ln \phi \rangle}) = \lim_{h \to 0} \frac{1}{h}(e^{\langle \mu, \ln u(h) \rangle} - e^{\langle \mu, \ln u(0) \rangle}) = \partial_t (e^{\langle \mu, \ln u(t) \rangle})\bigg|_{t = 0},
\end{equation*}
where we have used the duality of on/off branching Brownian motion for the second equation. Via integration under the integral sign, we have
\begin{equation*}
    \partial_t (e^{\langle \mu, \ln u(t) \rangle}) = e^{\langle \mu, \ln u(t) \rangle}\partial_t \langle \mu, \ln u(t) \rangle =  e^{\langle \mu, \ln u(t) \rangle} \langle \mu, \partial_t \ln u(t) \rangle = e^{\langle \mu, \ln u(t) \rangle} \langle \mu, \frac{1}{u(t)} \partial_t u(t) \rangle.
\end{equation*}
\end{proof}
In the next expositions we will write out the integration variables inside the angle brackets to explicitly state the form of the function that is being integrated.
\begin{prop}[Martingale problem for on/off binary branching Brownian motion]
    \label{prop:mpoobbm}
    We have that
    \begin{align*}
        M_t^\%(\phi) &:= \langle Z_t, \phi \rangle - \langle Z_0, \phi \rangle - \int_0^t \langle \; Z_s,\; 1_{\{i=1\}}(x,i) \bigl[\frac{1}{2}\Delta \phi(x,1) + c(\phi(x,0) - \phi(x,1))\bigr] \\
        &\qquad\qquad\qquad\qquad\qquad\qquad+ 1_{\{i=0\}}(x,i)\tilde{c}\bigl(\phi(x,1) -\phi(x,0)\bigr) \; \rangle ds
    \end{align*}
    is a martingale with quadratic variation
    \begin{align*}
        [M^\%(\phi)]_t &= 2 \int_0^t \langle \; Z_s, \; 1_{\{i=1\}}(x,i)\bigl[ \; \frac{1}{2} \nabla\phi(x,1)\cdot\nabla\phi(x,1) \; + \; \frac{\gamma}{2}(\Phi[e^{\phi(x,1)}]e^{-\phi(x,1)} -1) \\
        &\qquad\qquad\qquad\qquad\quad+\; c(e^{\phi(x,0) - \phi(x,1)} + \phi(x,1) - \phi(x,0) -1) \; \bigr] \\
        &\qquad\quad+ 1_{\{i=0\}}(x,i)\tilde{c}\bigl(e^{\phi(x,1) - \phi(x,0)} + \phi(x,0) - \phi(x,1) - 1\bigr) \;\rangle ds.
    \end{align*}
\end{prop}
\begin{proof}
This follows from Ito's formula and comparison of martingale parts as in Lemma 1.10 of \cite{E00}.
\end{proof}

\begin{theorem}[Tightness on path-space of on/off branching Brownian motion]
\label{thm:tpsoobbm}
    The on/off branching Brownian motion is tight on the path space.
\end{theorem}
\begin{proof}
    We have to check the compact containment condition and use the Aldous Rebolledo Theorem, see Theorem 1.17 of \cite{E00}, but this follows in the standard way and will therefore not be reiterated.
\end{proof}

\begin{theorem}[Generator of on/off super-Brownian motion]
\label{thm:goooSBM}
The generator of on/off super-Brownian motion on functions of the form $F_\phi: \mathcal{M}_F(\mathbb{R}^d \times \{0,1\}) \to \mathbb{R}, \; \mu \mapsto \exp(-\langle \mu, \phi \rangle)$ and $\phi \in bp\mathcal{C}(\mathbb{R}^d \times \{0,1\})$  is given by
\begin{equation*}
\mathcal{G}^\%F_\phi(\mu) = -\langle \mu, L\phi \rangle \exp(-\langle \mu, \phi \rangle),
\end{equation*}
where $L$ is the operator given by 
\begin{align*}
    \label{eq:gooSBM1}
Lu(t,(x,i)) &:= 1_{\{i=1\}} \bigl[ \; \frac{1}{2}\Delta u(t, (x,1)) - \frac{\gamma}{2}u(t, (x,1))^2 + \cact\bigl(u(t,(x,0)) - u(t,(x,1))\bigr) \; \bigr] \\
&\;+ 1_{\{i=0\}}\cdor\bigl(u(t,(x,1)) - u(t,(x,0))\bigr).
\end{align*}
\end{theorem}
\begin{proof}
From the on/off super-Brownian motion existence theorem, we have
\begin{equation*}
\mathbb{E}_\mu[e^{-\langle X_t, \phi \rangle}] = \exp(- \langle \mu, V_t \phi \rangle).
\end{equation*}
Keeping in mind that $V_0 \phi = \phi$, we get
\begin{align*}
\mathcal{G}^\%F_\phi(\mu) &= \lim_{h \to 0} \frac{1}{h}(\mathbb{E}_\mu[e^{-\langle X_h, \phi \rangle}] - e^{-\langle X_0, \phi \rangle})\\
&= \lim_{h \to 0} \frac{1}{h}(e^{-\langle \mu, V_h\phi \rangle}  - e^{-\langle \mu, \phi \rangle}) \\
&= \partial_t (e^{-\langle \mu, V_t\phi \rangle})\bigg|_{t=0} 
= -\langle \mu, L\phi \rangle e^{-\langle \mu, \phi \rangle}. 
\end{align*}
\end{proof}

\begin{remark}[Dormancy does not destroy polynomiality]
    Definitions 6.1 and 6.3 of \cite{2021_Cuchiero_MeasureValuedAffineAndPolynomialDiffusions} state, that an affine measure valued diffusion is a continous measure valued process, whose generator maps a Laplace kernel to the same Laplace kernel times an affine expression. Comparing this form of the generator with the generator of on/off super-Brownian motion, we see that the generator of on/off super-Brownian motion is of affine type (and therefore also a polynomial operator), which makes on/off super-Brownian motion an affine and therefore also a polynomial diffusion. Hence dormancy does not destroy the affine/polynomiality property.
    We will later use polynomiality to investigate the behaviour of the total mass process, which means that the polynomiality is also preserved under projections of the kind that map a measure-valued process to its total mass process.
\end{remark}
\begin{theorem}[Martingale problem for on/off super-Brownian motion]
    \label{prop:mpfoosbm}
    The process
    \begin{equation*}
        M_t(\phi) := \langle X_t, \phi \rangle - \langle X_0, \phi \rangle - \int_0^t \langle X_s, \tilde{L}\phi \rangle ds
    \end{equation*}
    is a martingale with quadratic variation
    \begin{equation*}
        [M(\phi)]_t = \gamma \int_0^t \langle X_s, 1_{\{ i= 1 \}}\phi^2 \rangle ds,
    \end{equation*}
    where $\tilde{L}$ is given by
    \begin{equation}
        \label{eq:mpooSBM-1}
    \tilde{L}\phi(x,i) := 1_{\{i=1\}}\bigl[\;\frac{1}{2}\Delta \phi(x,1) + c\bigl(\phi(x,0) - \phi(x,1)\bigr)\;\bigr] + 1_{\{i=0\}}\tilde{c}\bigl(\phi(x,1) - \phi(x,0)\bigr).
    \end{equation}
\end{theorem}
\begin{proof}
   The assertion follows again using Ito's formula and comparison of martingale parts as in Proposition \ref{prop:mpoobbm}.
\end{proof}
Having the martingale problem for on/off super-Brownian motion at hand, we can turn to the path continuity, which can be derived from  \cite[Theorem1.3, p.48-49]{1986_RoellyCoppoletta_ACriterionOfConvergenceOfMeasureValuedProcessesApplicationToMeasureBranchingProcesses}.

\begin{prop}[Path continuity of on/off super-Brownian motion]
\label{prop:pcoosbm}
The on/off super-Brownian motion has continuous paths.    
\end{prop}
\begin{proof}
    As in \cite[Proposition 2.15, p.\ 48]{E00}, path continuity follows by definition of weak convergence from the path continuity of $(\langle X_t, \phi \rangle)_{t \geq 0}$ for each $\phi$ in the convergence determining set of continuous bounded functions. 
\end{proof}

\begin{remark}
    One could also aim for Hölder continuity of paths by using martingale measures and the Kolmogorov continuity theorem as in \cite[Proposition 2.15, p.\ 48]{E00}. However, since we are interested in the convergence on the path space, we content ourselves with the continuity for now. In total there are (at least) three strategies for proving path continuity of on/off super-Brownian motion, namely: Via moments as in \cite{W68}, Theorem 3.1, via martingale measures giving Hölder continuity as in \cite{E00} or via martingale methods as in \cite{1986_RoellyCoppoletta_ACriterionOfConvergenceOfMeasureValuedProcessesApplicationToMeasureBranchingProcesses}.
\end{remark}	

\section{Total mass and long-term behaviour of on/off SBM}
\label{sec-final} 
\label{sec-total-mass-process} 

We now derive the total mass process of on/off SBM and employ it to investigate some long-term properties of our superprocess.

\begin{prop}
    \label{prop:tmoosbm}
    The total mass process of the active and the dormant component of on/off SBM, denoted by 
    $(\langle X^a_t, 1\rangle, \langle X^d_t, 1\rangle)_{t \ge 0}$,
    is equal in distribution to the solution $(p_t,q_t)_{t\ge 0}$ of the system of stochastic differential equations
    \begin{align} 
        d p_t = & (\tilde{c} q_t-c p_t)dt +\frac{\gamma}{2}\sqrt{p_t}dB_t, \notag\\
        d q_t = & (c p_t - \tilde{c} q_t) dt,
        \label{eq:ooFeller}
    \end{align}
	with initial conditions $p_0=\langle X^a_0, 1\rangle$ and $q_0=\langle X^d_0, 1\rangle$, where $(B_t)_{t\ge 0}$ is a standard Brownian motion. 
    Here, $(p,q)$ is called on/off Feller diffusion.
    Moreover, the total mass process $(\langle X_t, 1 \rangle)_{t\ge 0}$ is equal in law to the process $(r_t)_{t\ge 0}$ defined by 
    $r_t := p_t + q_t$ for $t\ge0,$ with initial conditions given as before. 
    \end{prop}
    Note that the generator of the above system, for suitable twice continuously differentiable test-functions $f:[0, \infty)^2\to \R$, takes the form
    \begin{equation}
        \label{eq:ooFellerGenerator}
        \mathcal{A}f(x,y) := (\tilde{c}y-cx) \partial_x f(x,y) + (cx-\tilde{c}y)\partial_y f(x,y) + \frac{\gamma}{2}x\partial_{xx}^2 f(x,y).
    \end{equation}
	Further, note that the on/off Feller diffusion admits a continuous strong solution (continuity of the total mass process already follows from the continuity of on/off SBM).
\begin{proof}
We aim to show that the two-dimensional moment generating functions of the two processes coincide, i.e.\ for $\theta_1, \theta_2$ defined in some neighbourhood of $0$, 
\begin{equation}
    \label{eq:tmoosbm1}
    \mathbb{E}_\mu\big[e^{-\theta_1\langle X^a_t, 1\rangle - \theta_2 \langle X^d_t, 1\rangle}\big] = \mathbb{E}_{(x,y)}\big[e^{-\theta_1 p_t - \theta_2 q_t}\big],
\end{equation}
for $X_0=\mu$ and $x=\langle \mu^a_0, 1\rangle, y = \langle \mu^d_0, 1\rangle$. Then, for each $t \geq 0$ it holds that
\begin{equation*}
    (\langle X^a_t, 1\rangle, \langle X^d_t, 1\rangle) \overset{d}{=} (p_t,q_t),
\end{equation*}
 and the uniqueness of the finite-dimensional distribution follows from the Markov property.

We have for the left-hand side of Equation \eqref{eq:tmoosbm1}, using the duality from Equation \eqref{eq:ooSBMduality} that
\begin{equation*}
    \mathbb{E}_\mu \big[e^{-\theta_1\langle X^a_t, 1\rangle - \theta_2 \langle X^d_t, 1\rangle}\big] = e^{-u(t)x - v(t)y} =: \tilde{\phi}_\theta(t,(x,y)),
\end{equation*}
where $u,v$ solve the system of ordinary differential equations given by
\begin{equation*}
    \begin{cases}
        u'(t) = -\frac{\gamma}{2} u^2(t) + c(v(t) - u(t)), \\
        v'(t) = \tilde c (u(t)- v(t)), \\
        u(0) = \theta_1, v(0) = \theta_2.
    \end{cases}
\end{equation*}
This can be obtained by setting the initial condition of the PDE-dual of on/off super-Brownian motion in Equations \eqref{eq:ooSBMPDEdual1} and \eqref{eq:ooSBMPDEdual2} to
$$\phi(x,i) := \begin{cases}
    \theta_1, i=1, \\
    \theta_2, i=0,
\end{cases}$$
which is constant in $x$, so that the Laplacian vanishes and the partial differential equation reduces to an ordinary differential equation.
The proof is finished by the unique solvability of the Kolmogorov forward equation corresponding to the generator $\mathcal{A}$ and by checking that $\tilde{\phi}_\theta$ indeed  solves it.
We leave this simple computation to the reader.

Finally, the equality in distribution of the total mass processes follows from the equality of the corresponding image measures under the  map $f(x,y):= x+y$.
\end{proof}
 
 \begin{remark}
 	Some care is needed when deriving the constants in Equation \eqref{eq:ooFeller}. Note that the drift terms for the on/off Feller diffusion differ from the drift terms of the seed bank Wright-Fisher diffusion in \cite{BGKW16} in their dependence on the parameters $c$ and $\tilde{c}$. This discrepancy also holds w.r.t.\ the deterministic dual of on/off SBM.
 \end{remark}

\begin{cor}[Long-term persistence]
    \label{cor:apm}
    We have that $r_t > 0$ for all $t \geq 0$ a.s.\ and therefore also $\langle X_t, 1 \rangle > 0$ for all $t \geq 0$ a.s.\
\end{cor}
This in particular means that on/off super-Brownian motion never dies out in finite time.
\begin{proof}
    Setting the influx of the seed bank component to zero in Equation \eqref{eq:ooFeller}, we consider the reduced system
    \begin{equation*}
        \begin{cases}
            d\tilde{q}_t = -cK\tilde{q}_t, \\
            \tilde{q}_0 = q_0,
        \end{cases}
    \end{equation*}
    whose solution is $\tilde{q}_t = q_0e^{-cKt} > 0$. 
By comparison, e.g.\ using Theorem 43.1 of \cite{2000_RogersWilliams_DiffusionsMarkovProcessesAndMartingalesVol2ItoCalculus}, we have
   \begin{equation*}
        r_t = p_t + q_t \geq q_t \geq \tilde{q}_t > 0, \quad \text{\;for all\;} t \geq 0\quad \mbox{ a.s.}
    \end{equation*}
 The corresponding statement for the total mass process $(\langle X_t, 1 \rangle)_{t \geq 0}$ follows by equality of its finite-dimensional distribution with $(r_t)_{t\ge 0}$ and path-continuity.
\end{proof}

Actually we can show more: Despite its persistence for all finite times, the total mass process eventually converges to zero as $t \to \infty$. For this we recall and employ some multi-type branching process machinery from \cite{2018_KyprianouPalau_ExtinctionPropertiesOfMultiTypeContinuousStateBranchingProcesses} taylored to our total mass process resp.\ the on/off Feller diffusion.
\begin{lemma}[Homogeneity of branching processes]
    \label{lem:hobo}
For $x \in \mathbb{R}_+$ and $\phi \in bp\mathcal{C}(\mathbb{R}^d \times \{0,1\})$ we have that
    \begin{equation*}
        \mathbb{E}_{x\mu}[\langle X_t, \phi\rangle] = x\mathbb{E}_{\mu}[\langle X_t, \phi\rangle].
    \end{equation*}
\end{lemma}
\begin{proof}
See Equations (10.3) and (11.2) in \cite{2008_Kyprianou_LevyProcessesAndContinuousStateBranchingProcessesPartIII}.
\end{proof}

Now, consider the mean-matrix
    \begin{equation*}
        [M(t)_{ij}] := \begin{bmatrix}
            \mathbb{E}_{(1,0)}[p_t] & \mathbb{E}_{(1,0)}[q_t] \\
            \mathbb{E}_{(0,1)}[p_t] & \mathbb{E}_{(0,1)}[q_t]
        \end{bmatrix} = \begin{bmatrix}
            \mathbb{E}_{\pi_1}[\langle 1, X_t^a \rangle] & \mathbb{E}_{\pi_1}[\langle 1 , X_t^d \rangle] \\
            \mathbb{E}_{\pi_0}[\langle 1, X_t^a \rangle] & \mathbb{E}_{\pi_0}[\langle 1, X_t^d \rangle]
        \end{bmatrix},
    \end{equation*}
    where the notation $\mathbb{E}_{(x,y)}[\cdot]$ indicates that we start $(p_t)_{t \geq 0}$ in $x$ and $(q_t)_{t \geq 0}$ in $y$, and $\pi_i$ is a probability measure concentrated on $\mathbb{R}^d \times \{i\}, i= 1,0$. 
\begin{lemma}
    \label{lem:mmtm}
    Let $f(i):=f(x,i)$ be a measurable function on $\mathbb{R}^d \times \{0,1\}$ that is constant in $x$. 
    Let $$\mathcal{M}_t[f](i) := \mathbb{E}_{\pi_i}[\langle X_t, f \rangle].
    $$ 
    Then, $\mathcal{M}_t$ has the semigroup property, which implies that the mean matrix also has the semigroup property, i.e.\ $M(s)M(t) = M(s+t)$.
	Furthermore, we have 
\begin{equation*}
    [M(t)f]_i = \mathcal{M}_t[f](i).
\end{equation*}
\end{lemma}
\begin{proof}
    See Equations (13) and (15) in \cite{2018_KyprianouPalau_ExtinctionPropertiesOfMultiTypeContinuousStateBranchingProcesses}.
\end{proof}
\begin{prop}
    \label{prop:tmtz}
    The total mass process of on/off super-Brownian motion converges almost surely to zero.
\end{prop}
\begin{proof}
    Let 
    \begin{equation*}
        \Hcal_{ij}(\lambda) := \int_0^\infty e^{\lambda t}M(t)_{ij} dt 
    \end{equation*}
    and let $g^{(x,y)}(t):= \mathbb{E}_{(x,y)}[p_t]$ and $h^{(x,y)}(t):= \mathbb{E}_{(x,y)}[q_t]$. From Equation \eqref{eq:ooFeller}, we get
    \begin{align*}
        g'(t) &= \cdor h(t) - \cact g(t), \\
        h'(t) &= \cact g(t) - \cdor h(t).
    \end{align*}
    The Laplace transform of this system is 
    \begin{align*}
        \lambda G(\lambda) - g(0) &= \cdor H(\lambda) - \cact G(\lambda), \\
        \lambda H(\lambda) - h(0) &= \cact G(\lambda) - \cdor H(\lambda).
    \end{align*}
    Therefore
\begin{equation*}
    G(\lambda) = \frac{(\lambda+ \cdor)g(0) + \cdor h(0)}{(\lambda+ \cdor)(\lambda+\cact) - \cdor\cact} < \infty
\end{equation*}
and
\begin{equation*}
      H(\lambda) = \frac{\cact g(0) + (\lambda+\cact)h(0)}{(\lambda+\cact)(\lambda+ \cdor) - \cact\cdor} < \infty,
\end{equation*}
if $\lambda \neq 0$. We obtain
\begin{equation*}
    \Hcal(\lambda) = \frac{1}{(\lambda+ \cdor)(\lambda+\cact) - \cdor\cact} \begin{bmatrix}
        (\lambda+ \cdor) & \cact\\
        \cdor & (\lambda+\cact) \\
    \end{bmatrix}.
\end{equation*}
Define, for $t\ge 0$,
\begin{equation*}
    W_t := e^{\lambda t} \langle X_t, f^{(j)}_\lambda \rangle,
\end{equation*}
where $f^{(j)}_\lambda$ is the $j$-th column of the matrix 
$\Hcal(\lambda)$. More explicitly, we define $f^{(j)}_\lambda$ to be the function on $\mathbb{R}^d \times \{0,1\}$ given by
\begin{equation*}
    f^{(j)}_\lambda(x,i) := \Hcal_{1j}(\lambda) 1_{\mathbb{R}^d \times \{1\}}(x,i) + \Hcal_{2j}(\lambda)1_{\mathbb{R}^d \times \{0\}}(x,i),
\end{equation*}
which is constant in $x$. 
Note that for $\lambda > 0$ we have $f^{(j)}_\lambda > 0$.

Next we show that $(W_t)_{t \geq 0}$ is a supermartingale w.r.t.\ its natural filtration (cf.\ \cite[Proposition~3]{2018_KyprianouPalau_ExtinctionPropertiesOfMultiTypeContinuousStateBranchingProcesses}). We have
\begin{equation}
    \label{eq:tmtz2}
\mathbb{E}_x[W_{t+s} | \mathcal{F}_s] = \mathbb{E}_x[ e^{\lambda (t+s)} \langle X_{t+s}, f^{(j)}_\lambda \rangle | \mathcal{F}_s] = e^{\lambda (t+s)} \mathbb{E}_x[ \langle X_{t+s}, f^{(j)}_\lambda \rangle | \mathcal{F}_s] = e^{\lambda (t+s)} \mathbb{E}_{X_s}[ \langle X_t, f^{(j)}_\lambda \rangle],
\end{equation}
by the Markov property. Now let $\bar{X_s^a}$ (respectively $\bar{X_s^d}$) be the active (dormant) component embedded as a measure on $\mathbb{R}^d \times \{0,1\}$. Then we have
\begin{align*}
    \mathbb{E}_{X_s}[ \langle X_t, f^{(j)}_\lambda \rangle] &= \mathbb{E}_{\bar{X_s^a}}[\langle X_t, f_\lambda^{(j)} \rangle] + \mathbb{E}_{\bar{X_s^d}}[\langle X_t, f_\lambda^{(j)} \rangle] \\
    &= \langle X_s^a, 1 \rangle \mathbb{E}_{\frac{1}{\langle X_s^a, 1 \rangle}\bar{X_s^a}}[\langle X_t^a, 1 \rangle]f_\lambda^{(j)}(1) + \langle X_s^a, 1 \rangle \mathbb{E}_{\frac{1}{\langle X_s^a, 1 \rangle}\bar{X_s^a}}[\langle X_t^d, 1 \rangle]f_\lambda^{(j)}(0) \\
    & + \langle X_s^d, 1 \rangle \mathbb{E}_{\frac{1}{\langle X_s^d, 1\rangle}\bar{X_s^d}}[\langle X_t^a, 1 \rangle]f_\lambda^{(j)}(1) + \langle X_s^d, 1 \rangle \mathbb{E}_{\frac{1}{\langle X_s^d, 1\rangle}\bar{X_s^d}}[\langle X_t^d, 1 \rangle]f_\lambda^{(j)}(0) \\
    &= \begin{bmatrix}
        \langle X_s^a, 1\rangle & \langle X_s^d, 1\rangle
    \end{bmatrix} M(t)f^{(j)}_\lambda,
\end{align*}
where we have used the branching property and the homogeneity of branching processes from Lemma \ref{lem:hobo}.
Let $[\cdot]_i$ denote the $i$-th component of a vector. Now we have for $i \in \{1,2\}$, that
\begin{align*} 
    [M(t)f^{(j)}_\lambda]_i &= \int_0^\infty e^{\lambda s}[M(t)M(s)]_{ij} ds \\
    &= \int_0^\infty e^{\lambda s}[M(t+s)]_{ij} ds \\
    &= e^{-\lambda t}\int_t^\infty \underbrace{e^{\lambda s}[M(s)]_{ij}}_{\geq 0} ds \leq e^{-\lambda t}\int_0^\infty e^{\lambda s}[M(s)]_{ij} ds =  e^{-\lambda t} \Hcal_{ij}(\lambda) = e^{-\lambda t} [f^{(j)}_\lambda]_i,
\end{align*}
where we have used Lemma \ref{lem:mmtm}. We obtain
\begin{align*}
    e^{\lambda (t+s)} \mathbb{E}_{X_s}[ \langle X_t, f^{(j)}_\lambda \rangle] &= e^{\lambda (t+s)} \begin{bmatrix}
        \langle X_s^a, 1\rangle & \langle X_s^d, 1\rangle
    \end{bmatrix} M(t)f^{(j)}_\lambda \\
    &\leq e^{\lambda (t+s)} \begin{bmatrix}
        \langle X_s^a, 1\rangle & \langle X_s^d, 1\rangle
    \end{bmatrix} e^{-\lambda t} f^{(j)}_\lambda \\
    &= e^{\lambda s} \begin{bmatrix}
        \langle X_s^a, 1\rangle & \langle X_s^d, 1\rangle
    \end{bmatrix} f^{(j)}_\lambda = W_s,
\end{align*}
which together with Equation \eqref{eq:tmtz2} is the supermartingale property of $W$.

Now, since $W$ is a positive supermartingale, it converges almost surely to some \emph{finite} random variable. Since $\lambda > 0$, we have that $e^{\lambda t} \overset{t \to \infty}{\longrightarrow} \infty$ and consequently $\langle X_t, f^{(j)}_\lambda \rangle$ converges to zero. Since $f^{(j)}_\lambda > 0$, this further implies that the total active mass process and the total dormant mass process converge to zero.
\end{proof}

We now show that, despite the long-term persistence of the total mass process, with positive probability there are times at which the active component becomes extinct. In other words: The origin is accessible for the active population, but inaccessible for the dormant population.

\begin{prop}[Active component hitting zero with positive probability]
    \label{prop:achz}
    Let $0 \leq y \leq \frac{1}{2\cdor}$.
    Then for any $T > 0$, there exists some $\epsilon > 0$ such that for all $(p_0, q_0)$ with $\|(p_0, q_0) - (0, y) \| < \epsilon$ we have
    \begin{equation*}
        \mathbb{P}_{(p_0, q_0)}[p_t = 0 \; \text{for some} \; t \leq T] > 0.
    \end{equation*}
\end{prop}
In particular,
if we start the total mass process close enough to zero in the active component and with a suitably small seed bank such that the resuscitating particles do {not} push the active component away from zero too quickly, then the active component may completely vanish with positive probability and the population is only surviving due to the seed bank effect.

\begin{proof}
    We note that the on/off Feller diffusion is a polynomial diffusion and check the conditions of Theorem 5.7 (iii) in \cite{2016_FilipovicLarsson_PolynomialDiffusionsAndApplicationsInFinance}, see Theorem \ref{thm:flpbc} from the appendix, which then delivers the result. In the notation of Theorem \ref{thm:flpbc}, we have for the dynamics of the total mass processes that $d = 2$ and
    \begin{equation*}
        E:=\mathbb{R}_+^2 = \bigcap_{p \in \mathcal{P}} \{x \in \mathbb{R}^2 \, | \, p(x) \geq 0\},
    \end{equation*}
where $\mathcal{P} := \{p_1, p_2\}$ with $p_1(x_1, x_2) = x_1$ and $p_2(x_1,x_2) = x_2$.

For the diffusion coefficient we obtain
\begin{equation*}
    a: (x_1, x_2) \mapsto \sigma \sigma^T (x_1,x_2) = \begin{bmatrix}
    x_1 & 0 \\
    0 & 0    
    \end{bmatrix}.
\end{equation*}
We can assume that $(p_t)_{t \geq 0}$ satisfies Condition \eqref{eq:zerotimeinzero}, because otherwise it would hit zero and the proof would be finished.

Next we have for all $(x_1, x_2) \in \mathbb{R}^2$ that
\begin{equation*}
    a(x_1, x_2) \nabla p_1(x_1, x_2) = \begin{bmatrix}
        x_1 & 0 \\
        0 & 0
    \end{bmatrix}\begin{bmatrix}
        1 \\ 
        0
    \end{bmatrix} = \underbrace{\begin{bmatrix}
        1 \\
        0
    \end{bmatrix}}_{=: h} p_1(x_1, x_2).
\end{equation*}
To apply the theorem in  \cite{2016_FilipovicLarsson_PolynomialDiffusionsAndApplicationsInFinance}, we have to choose $\bar{x} := (x_1, x_2) \in E \cap \{p_1 = 0\}$ in such a way that
\begin{equation*}
    \mathcal{G}p_1(x_1, x_2) := \mathcal{A}p_1(x_1, x_2) = (\cdor x_2 - \cact x_1) \geq 0
\end{equation*} 
and
\begin{equation*}
    2\mathcal{G}p(x_1, x_2) = 2(\cdor x_2 - \cact x_1) \leq h^T \nabla p_1(x_1,x_2) = \begin{bmatrix}
        1 & 0
    \end{bmatrix}\begin{bmatrix}
        1 \\
        0
    \end{bmatrix} = 1,
\end{equation*}
where $\mathcal{A}$ is the generator of the on/off Feller diffusion. Since $(x_1, x_2) \in \{p_1 = 0\}$, we have $x_1 = 0$, which leaves $x_2 \in [0, \frac{1}{2\cdor}]$, but this is satisfied by assumption. The assertion then follows by application of Theorem 5.7 (iii) of \cite{2016_FilipovicLarsson_PolynomialDiffusionsAndApplicationsInFinance}.
\end{proof}

Note that since our total mass process is two-dimensional, we cannot employ the usual speed-measure scale-function formalism for one-dimensional diffusions in order to clarify its boundary behaviour. The above machinery for polynomial diffusions derived in \cite{2016_FilipovicLarsson_PolynomialDiffusionsAndApplicationsInFinance} instead relies (in our particular case) on constructing a suitable comparison with a BESQ-process.

The result can be slightly strengthened. Indeed, since the total mass converges to zero almost surely, at some point in time the active and dormant component will run into the neighbourhood of zero from Proposition \ref{prop:achz}. Restarting the process at this time will then make the active component hit zero with positive probability.
\begin{cor}[Active component hitting zero from any starting point wpp]
	\label{cor:acehzwpp}
    Let $(p_0,q_0)$ be an arbitrary initial condition in $[0,\infty)^2$. Then,
    \begin{equation*}
        \mathbb{P}_{(p_0, q_0)}[p_t = 0 \; \text{for some} \; t \geq 0] > 0.
    \end{equation*}
\end{cor}

\begin{proof}
    Set $y = 0$ and fix some $T \geq 0$ in Proposition \ref{prop:achz}. Let $\tilde{\epsilon} := \frac{\epsilon}{2}$, where $\epsilon$ is also given by Proposition \ref{prop:achz}. Let $\tau := \inf \{t \geq 0 \; | \; \|(p_t, q_t)\| \leq \tilde{\epsilon} \}$ be the first hitting time of the $\tilde{\epsilon}$-ball around zero.
    Since the total mass converges to zero almost surely, we have that $\tau < \infty$ almost surely.
    Using path-continuity, the strong Markov property and applying Proposition \ref{prop:achz} gives
    \begin{equation*}
    	\mathbb{P}_{(p_0, q_0)}[p_t = 0 \; \text{for some} \; t \geq 0] \geq \mathbb{P}_{(p_\tau, q_\tau)}[p_t = 0 \; \text{for some} \; t \geq 0] >  0,
    \end{equation*}
	as desired.
\end{proof}
	 
	\section*{Acknowledgments}
	This work has been supported by DFG IRTG 2544 Berlin-Oxford and by DFG under Germany's Excellence Strategy – The Berlin Mathematics Research Center MATH+ (EXC-2046/1, project ID 390685689, BMS Stipend).
	The second author also thanks Goethe University Frankfurt for hospitality in November 2022, where part of this research was carried out. 
	\iflists
	\glossarystyle{super}
	\setlength{\glsdescwidth}{0.8\textwidth}
	\glsaddall
	\printglossary[title=Symbols,toctitle=Symbols, nonumberlist]

	\listoftheorems[show={defn}]
	\fi

	\appendix
	\section{Localization lemma}
\begin{lemma}[Localization Lemma]
    \label{lem:ll}
    Let $t_0 < t_1 < ... < t_n=t$ be an arbitrary partition of $[t_0,t]$, then
    \begin{equation}
        \label{eq:lla}
    \underbrace{v^r(x) + \mathbb{E}_{r,x}[\int_r^t \psi^s(v^s)(\xi_s) ds]}_{=: U_t^r(x)} = \mathbb{E}_{r,x}[f(\xi_t)],
    \end{equation}
    for $r \in [t_0,t]$ if and only if,
    \begin{equation}
        \label{eq:llb1}
    v^t(x) = f(x)
    \end{equation}
    and
    \begin{equation}
        \label{eq:llb2}
    \underbrace{v^r(x) + \mathbb{E}_{r,x}[\int_r^{t_i} \psi^s(v^s)(\xi_s) ds]}_{= U_{t_i}^r(x)} = \mathbb{E}_{r,x}[v^{t_i}(\xi_{t_i})],
    \end{equation}
    for $r \in [t_{i-1}, t_i)$ and all $i=1,...n$.
    \end{lemma}
    See Lemma 3.2 of \cite{Dynkin1994}.
\begin{proof}
    First we show an auxiliary result, namely
    \begin{equation}
        \label{eq:ll1}
        U_t^r(x) = U_{t_i}^r(x) + \mathbb{E}_{r,x}[U_t^{t_i}(\xi_{t_i})] - \mathbb{E}_{r,x}[v^{t_i}(\xi_{t_i})], \quad r < t_i.
    \end{equation}
    By definition and using the Markov property we have
    \begin{align*}
        &U_{t_i}^r(x) + \mathbb{E}_{r,x}[U_t^{t_i}(\xi_{t_i})] - \mathbb{E}_{r,x}[v^{t_i}(\xi_{t_i})] \\
        &= v^r(x) + \mathbb{E}_{r,x}[\int_r^{t_i} \psi(v^s)(\xi_s) ds] + \mathbb{E}_{r,x}[v^{t_i}(\xi_{t_i}) + \mathbb{E}_{t_i,\xi_{t_i}}[\int_{t_i}^r \psi(v^s)(\xi_s)ds]] - \mathbb{E}_{r,x}[v^{t_i}(\xi_{t_i})] \\
        &= v^r(x) + \mathbb{E}_{r,x}[\int_r^t \psi(v^s)(\xi_s) ds] = U_t^r(x).
    \end{align*}
    Now suppose that Equation \eqref{eq:lla} holds. First we note that
    \begin{equation}
        \label{eq:llc}
        \mathbb{E}_{r,x}[U_t^{t_i}(\xi_{t_i})] = U_t^r(x),
    \end{equation}
    because
    \begin{align*}
        \mathbb{E}_{r,x}[U_t^{t_i}(\xi_{t_i})] &= \mathbb{E}_{r,x}[v^{t_i}(\xi_{t_i}) + \mathbb{E}_{t_i, \xi_{t_i}}[\int_{t_i}^t \psi(v^s)(\xi_s)ds]] \\
        &= \mathbb{E}_{r,x}[\mathbb{E}_{t_i, \xi_{t_i}}[f(\xi_t) - \int_{t_i}^t \psi(v^s)(\xi_s) ds]] + \mathbb{E}_{r,x}[\int_{t_i}^t \psi(v^s)(\xi_s)ds] \\
        &= \mathbb{E}_{r,x}[f(\xi_t) - \int_{t_i}^t \psi(v^s)(\xi_s) ds] + \mathbb{E}_{r,x}[\int_{t_i}^t \psi(v^s)(\xi_s)ds] \\
        &= \mathbb{E}_{r,x}[f(\xi_t)] = U_t^r(x),
    \end{align*}
    where the second equality follows by setting $r:= t_i$ and $x := \xi_{t_i}$ in Equation \eqref{eq:lla} and from the Markov property. For the last equality we have used Equation \eqref{eq:lla} again.
Then we have for $i=1,...,n$ from Equations \eqref{eq:ll1} and \eqref{eq:llc} that
\begin{align*}
    U_t^r(x) &= U_{t_i}^r(x) + \mathbb{E}_{r,x}[U_t^{t_i}(\xi_{t_i})] - \mathbb{E}_{r,x}[v^{t_i}(\xi_{t_i})] \\
    &= U_{t_i}^r(x) + U_t^r(x) - \mathbb{E}_{r,x}[v^{t_i}(\xi_{t_i})].
\end{align*}
Rearrangeing yields
\begin{equation*}
    U_{t_i}^r(x) = \mathbb{E}_{r,x}[v^{t_i}(\xi_{t_i})],
\end{equation*}
which is Equation \eqref{eq:llb2}. Equation \eqref{eq:llb1} immediately follows from Equation \ref{eq:lla} by setting $r := t$.

    For the converse direction suppose that Equations \eqref{eq:llb1} and \eqref{eq:llb2} hold. We proceed by backwards induction over $i$. Let $i := n$, then by Equation \eqref{eq:llb2} we have that Equation \eqref{eq:lla} holds for $r \in [t_{n-1}, t)$.
    Next we show, that if Equation \eqref{eq:lla} holds for $r \in [t_i, t)$, then it follows, that Equation \eqref{eq:lla} holds for $r \in [t_{i-1}, t_i)$. And therefore Equation \eqref{eq:lla} will also hold for $r \in [t_{i-1}, t)$.

    Suppose Equation \eqref{eq:lla} holds for $r \in [t_i, t)$, i.e.\ $U_t^{t_i}(x) = \mathbb{E}_{t_i,x}[f(\xi_t)]$. By setting $s:=t_i$ in Equation \eqref{eq:ll1} we have for $r < t_i$ that
    \begin{align*}
U_t^r(x) &= U_{t_i}^r(x) + \mathbb{E}_{r,x}[U_t^{t_i}(\xi_{t_i})] - \mathbb{E}_{r,x}[v^{t_i}(\xi_{t_i})] \\
&= U_{t_i}^r(x) + \mathbb{E}_{r,x}[\mathbb{E}_{t_i,\xi_i}[f(\xi_t)]] - U_{t_i}^r(x) \\
&= \mathbb{E}_{r,x}[f(\xi_t)],
    \end{align*}
    where we have used Equations \eqref{eq:lla} and \eqref{eq:llb2}.
\end{proof}
		
	\section{Polynomial boundary classification}
	
\begin{theorem}[Polynomial boundary classification]
    \label{thm:flpbc}
Let $\mathcal{P}$ be a collection of polynomials on $\mathbb{R}^d$. Define $E := \{x \in \mathbb{R}^d \;|\; p(x) \geq 0 \text{\;for all\;} p \in \mathcal{P}\}$.
Let $a: \mathbb{R}^d \to \mathbb{S}^d$ take values in the semi-definite matrices such that $a_{ij}$ is a polynomial of degree at most two and $b: \mathbb{R}^d \to \mathbb{R}^d$ with $b_i$ being a polynomial of degree at most one. Let $\sigma$ be given by $\sigma\sigma^T=a$ and let $X$ be an $E$-valued solution of
\begin{equation*}
    dX_t = b(X_t)dt + \sigma(X_t)dB_t,
\end{equation*}
for some $d$-dimensional Brownian motion $B$ and let $X$ satisfy
\begin{equation}
    \label{eq:zerotimeinzero}
    \int_0^t 1_{\{p(X_s) = 0\}} ds = 0
\end{equation}
for all $t \geq 0$ and $p \in \mathcal{P}$. Denote by $\mathcal{G}$ the generator of $X$.

Now, fix $p \in \mathcal{P}$ and $\bar{x} \in E \cap \{p=0\}$, where $\{p=0\}$ is shorthand notation for $\{x \in \mathbb{R}^d |\;p(x)=0\}$, and let $h$ be a vector of polynomials such that $a(x)\nabla p(x) = h(x) p(x)$ for all $x \in \mathbb{R}^d$. If $\mathcal{G}p(\bar{x}) \geq 0$ and $2\mathcal{G}p(\bar{x}) - h(\bar{x})^T \nabla p(\bar{x}) < 0$, then for any $T > 0$ there exists some $\epsilon > 0$ such that if $\|X_0 - \bar{x}\| < \epsilon$ a.s., then $p(X_t) = 0$ for some $t < T$ with positive probability. 
\end{theorem}
\begin{proof}
    See Equations (2.2), (5.1), (5.3) and Theorem 5.7 (iii) in \cite{2016_FilipovicLarsson_PolynomialDiffusionsAndApplicationsInFinance}.
\end{proof}

	\bibliographystyle{alpha}
	\bibliography{o-o-SBM}

\begin{thebibliography}{BGCKWB20}

\bibitem[BGCKWB16]{BGKW16}
J.~Blath, A.~Gonz{\'a}lez~Casanova, N.~Kurt, and M.~Wilke-Berenguer.
\newblock A new coalescent for seed-bank models.
\newblock {\em Ann. Appl. Probab.}, 26(2):857--891, 2016.

\bibitem[BGCKWB20]{BGKW20}
J.~Blath, A.~Gonzalez~Casanova, N.~Kurt, and M.~Wilke-Berenguer.
\newblock The seed bank coalescent with simultaneous switching.
\newblock {\em Electron. J. Probab.}, 25:21, 2020.
\newblock Id/No 27.

\bibitem[BHJN21]{BHJN22+}
J.~Blath, M.~Hammer, D.~Jacobi, and F.~Nie.
\newblock How the interplay of dormancy and selection affects the wave of
  advance of an advantageous gene.
\newblock {\em Preprint}, 2021.

\bibitem[BHN22]{BHN22}
J.~Blath, M.~Hammer, and F.~Nie.
\newblock The stochastic fisher-kpp equation with seed bank and on/off
  branching coalescing brownian motion.
\newblock {\em Stochastic and partial differential equations -- Analysis and
  computations}, 2022.

\bibitem[Bra78]{1978_Bramson_MaximalDisplacementOfBranchingBrownianMotion}
M.~D. Bramson.
\newblock Maximal displacement of branching {Brownian} motion.
\newblock {\em Commun. Pure Appl. Math.}, 31:531--581, 1978.

\bibitem[Bra83]{1983_Bramson_ConvergenceOfSolutionsOfTheKolmogorovEquationToTravellingWaves}
M.~D. Bramson.
\newblock {\em Convergence of solutions of the {Kolmogorov} equation to
  travelling waves}, volume 285 of {\em Mem. Am. Math. Soc.}
\newblock Providence, RI: American Mathematical Society (AMS), 1983.

\bibitem[CDE18]{CDE18}
J.~A. Chetwynd-Diggle and A.~M. Etheridge.
\newblock Superbrownian motion and the spatial lambda-{Fleming}-{Viot} process.
\newblock {\em Electron. J. Probab.}, 23:36, 2018.
\newblock Id/No 71.

\bibitem[CGdPSF21]{2021_Cuchiero_MeasureValuedAffineAndPolynomialDiffusions}
C.~Cuchiero, F.~Guida, L.~di~Persio, and S.~Svaluto-Ferro.
\newblock Measure-valued affine and polynomial diffusions.
\newblock {\em Preprint}, 2021.

\bibitem[CLSF19]{CLS19}
C.~Cuchiero, M.~Larsson, and S.~Svaluto-Ferro.
\newblock Probability measure-valued polynomial diffusions.
\newblock {\em Electron. J. Probab.}, 24:32, 2019.
\newblock Id/No 30.

\bibitem[CP20]{CP20}
J.~T. Cox and E.~A. Perkins.
\newblock Rescaling the spatial {Lambda}-{Fleming}-{Viot} process and
  convergence to super-{Brownian} motion.
\newblock {\em Electron. J. Probab.}, 25:56, 2020.
\newblock Id/No 57.

\bibitem[Daw75]{D75}
D.~A. Dawson.
\newblock Stochastic evolution equations and related measure processes.
\newblock {\em J. Multivariate Anal.}, 5:1--52, 1975.

\bibitem[Daw93]{D93}
D.~A. Dawson.
\newblock Measure-valued {Markov} processes.
\newblock In {\em Ecole d'Et\'e de probabilit\'es de Saint-Flour XXI - 1991, du
  18 Ao\^ut au 4 Septembre, 1991}, pages 1--260. Berlin: Springer-Verlag, 1993.

\bibitem[DGL02]{2002_DawsonGorostizaLi_NonlocalBranchingSuperprocessesAndSomeRelatedModels}
D.~A. Dawson, L.~G. Gorostiza, and Z.~Li.
\newblock Nonlocal branching superprocesses and some related models.
\newblock {\em Acta Appl. Math.}, 74(1):93--112, 2002.

\bibitem[DH82]{DH82}
D.~A. Dawson and K.~J. Hochberg.
\newblock Wandering random measures in the {Fleming}-{Viot} model.
\newblock {\em Ann. Probab.}, 10:554--580, 1982.

\bibitem[DK99]{DK99}
P.~Donnelly and T.~G. Kurtz.
\newblock Particle representations for measure-valued population models.
\newblock {\em Ann. Probab.}, 27(1):166--205, 1999.

\bibitem[DLG02]{2002_DuquesneLeGall_RandomTreesLevyProcessesAndSpatialBranchingProcesses}
Thomas Duquesne and J.~F. Le~Gall.
\newblock {\em Random trees, {L{\'e}vy} processes and spatial branching
  processes}, volume 281 of {\em Ast{\'e}risque}.
\newblock Paris: Soci{\'e}t{\'e} Math{\'e}matique de France, 2002.

\bibitem[DLG05]{2005_DuquesneLeGall_ProbabilisticAndFractalAspectsOfLevyTrees}
T.~Duquesne and J.-F. Le~Gall.
\newblock Probabilistic and fractal aspects of {L{\'e}vy} trees.
\newblock {\em Probab. Theory Relat. Fields}, 131(4):553--603, 2005.

\bibitem[{Dyn}94]{Dynkin1994}
E.~B. {Dynkin}.
\newblock {\em {An introduction to branching measure-valued processes}},
  volume~6.
\newblock Providence, RI: American Mathematical Society, 1994.

\bibitem[EK19]{2019_EtheridgeKurtz_GenealogicalConstructionsOfPopulationModels}
A.~M. Etheridge and T.~G. Kurtz.
\newblock Genealogical constructions of population models.
\newblock {\em Ann. Probab.}, 47(4):1827--1910, 2019.

\bibitem[Eth00]{E00}
A.~M. Etheridge.
\newblock {\em An introduction to superprocesses}, volume~20 of {\em Univ.
  Lect. Ser.}
\newblock Providence, RI: AMS, American Mathematical Society, 2000.

\bibitem[FL16]{2016_FilipovicLarsson_PolynomialDiffusionsAndApplicationsInFinance}
D.~Filipovi{\'c} and M.~Larsson.
\newblock Polynomial diffusions and applications in finance.
\newblock {\em Finance Stoch.}, 20(4):931--972, 2016.

\bibitem[FV79]{FV79}
W.~H. Fleming and M.~Viot.
\newblock Some measure-valued {Markov} processes in population genetics theory.
\newblock {\em Indiana Univ. Math. J.}, 28:817--843, 1979.

\bibitem[GCKT21]{GKT21}
A.~Gonz{\'a}lez~Casanova, N.~Kurt, and A.~T{\'o}bi{\'a}s.
\newblock Particle systems with coordination.
\newblock {\em ALEA, Lat. Am. J. Probab. Math. Stat.}, 18(2):1817--1844, 2021.

\bibitem[GdHO22]{GHO22}
A.~Greven, F.~den Hollander, and M.~Oomen.
\newblock Spatial populations with seed-bank: well-posedness, duality and
  equilibrium.
\newblock {\em Electron. J. Probab.}, 27:88, 2022.
\newblock Id/No 18.

\bibitem[GLM90]{GorostizaLopezMimbela1990}
L.~G. Gorostiza and J.~A. L\'{o}pez-Mimbela.
\newblock The multitype measure branching process.
\newblock {\em Adv. in Appl. Probab.}, 22(1):49--67, 1990.

\bibitem[Isc88]{I88}
I.~Iscoe.
\newblock On the supports of measure-valued critical branching {Brownian}
  motion.
\newblock {\em Ann. Probab.}, 16(1):200--221, 1988.

\bibitem[JK14]{2014_JansenKurt_OnTheNotionsOfDualityForMarkovProcesses}
Sabine Jansen and Noemi Kurt.
\newblock On the notion(s) of duality for {Markov} processes.
\newblock {\em Probab. Surv.}, 11:59--120, 2014.

\bibitem[KP18]{2018_KyprianouPalau_ExtinctionPropertiesOfMultiTypeContinuousStateBranchingProcesses}
A.~E. Kyprianou and S.~Palau.
\newblock Extinction properties of multi-type continuous-state branching
  processes.
\newblock {\em Stochastic Processes Appl.}, 128(10):3466--3489, 2018.

\bibitem[KR11]{2011_KurtzRodriguez_PoissonRepresentationsOfBranchingMarkovAndMeasureValuedBranchingProcesses}
T.~G. Kurtz and E.~R. Rodrigues.
\newblock Poisson representations of branching {Markov} and measure-valued
  branching processes.
\newblock {\em Ann. Probab.}, 39(3):939--984, 2011.

\bibitem[KS88]{KS88}
N.~Konno and T~Shiga.
\newblock Stochastic partial differential equations for some measure-valued
  diffusions.
\newblock {\em Probability Theory and Related Fields}, 79:201--225, 1988.

\bibitem[Kyp08]{2008_Kyprianou_LevyProcessesAndContinuousStateBranchingProcessesPartIII}
A.~E. Kyprianou.
\newblock Lévy processes and continuous-state branching processes: part iii,
  2008.

\bibitem[LdHWBB21]{LHWB21}
J.~T. Lennon, F.~den Hollander, M.~Wilke-Berenguer, and J.~Blath.
\newblock Principles of seed banks and the emergence of complexity from
  dormancy.
\newblock {\em Nature Communications}, 12(1), Aug 10 2021.

\bibitem[LG91]{1991_LeGall_BrownianExcursionsTreesAndMeasureValuedBranchingProcesses}
J.-F. Le~Gall.
\newblock Brownian excursions, trees and measure-valued branching processes.
\newblock {\em Ann. Probab.}, 19(4):1399--1439, 1991.

\bibitem[McK75]{1975_McKean_ApplicationOfBrownianMotionToTheEquationOfKolmogorovPetrovskiiPiskunov}
H.~P. McKean.
\newblock Application of {Brownian} motion to the equation of
  {Kolmogorov}-{Petrovskii}- {Piskunov}.
\newblock {\em Commun. Pure Appl. Math.}, 28:323--331, 1975.

\bibitem[Per92]{Perkins1992}
E.~A. Perkins.
\newblock {\em Conditional Dawson---Watanabe Processes and Fleming---Viot
  Processes}, pages 143--156.
\newblock Birkh{\"a}user Boston, Boston, MA, 1992.

\bibitem[PR21]{PR21}
N.~Perkowski and T.~Rosati.
\newblock A rough super-{Brownian} motion.
\newblock {\em Ann. Probab.}, 49(2):908--943, 2021.

\bibitem[RC86]{1986_RoellyCoppoletta_ACriterionOfConvergenceOfMeasureValuedProcessesApplicationToMeasureBranchingProcesses}
S.~Roelly-Coppoletta.
\newblock A criterion of convergence of measure-valued processes: {Application}
  to measure branching processes.
\newblock {\em Stochastics}, 17:43--65, 1986.

\bibitem[RW00]{2000_RogersWilliams_DiffusionsMarkovProcessesAndMartingalesVol2ItoCalculus}
L.~C.~G. Rogers and D.~Williams.
\newblock {\em Diffusions, {Markov} processes, and martingales. {Vol}. 2:
  {It{\^o}} calculus.}
\newblock Cambridge: Cambridge University Press, 2nd edition, 2000.

\bibitem[Wat68]{W68}
S.~Watanabe.
\newblock A limit theorem of branching processes and continuous state branching
  processes.
\newblock {\em J. Math. Kyoto Univ.}, 8:141--167, 1968.

\end{thebibliography}
	
\end{document}